\pdfoutput=1
\documentclass[12pt]{amsart}
\usepackage[T1]{fontenc}
\usepackage[utf8]{inputenc}
\usepackage{amssymb}
\usepackage{amsthm}
\usepackage{longtable}
\usepackage{amsmath}
\usepackage{mathrsfs}
\usepackage{nicematrix}
\usepackage{yfonts}
\usepackage{mathtools}
\usepackage{tikz-cd}
\usepackage{xcolor}
\usepackage{fullpage}
\usepackage{graphicx}
\usepackage{biblatex}
\usepackage[parfill]{parskip}
\usepackage{hyperref}
\hypersetup{
  colorlinks   = true,    
  urlcolor     = black,    
  linkcolor    = black,    
  citecolor    = black      
}
\newtheorem{theorem}{Theorem}

\newtheorem{proposition}[theorem]{Proposition}
\newtheorem{lemma}[theorem]{Lemma}

\newtheorem{observation}[theorem]{Observation}
\newtheorem{definition}[theorem]{Definition}
\newtheorem{corollary}[theorem]{Corollary}

\numberwithin{theorem}{section}

\DeclarePairedDelimiter\floor{\lfloor}{\rfloor}

\newcommand\rr{\rightarrow}
\newcommand{\PP}{\mathbb P}
\newcommand{\ZZ}{\mathbb{Z}}
\newcommand{\QQ}{\mathbb{Q}}

\newcommand{\CC}{\mathbb{C}}

\newcommand{\mc}{\mathcal}

\newcommand{\wb}[1]{\overline{#1}}
\newcommand{\wdt}{\widetilde}
\newcommand\blfootnote[1]{%
  \begingroup
  \renewcommand\thefootnote{}\footnote{#1}%
  \addtocounter{footnote}{-1}%
  \endgroup
}

\title{Semistable degenerations of double octics}

\thanks{The author was partially supported by National Science Centre, Poland, grant No. 2020/39/B/ST1/03358.}
\author[]{Marcin Oczko}

\begin{document}

\maketitle

\blfootnote{Marcin Oczko: Jagiellonian University, Faculty of Mathematics and Computer Science, Cracow, Poland, email: marcin.oczko@im.uj.edu.pl}


\begin{abstract}
We present an algorithm for computing semistable degeneration of double octic Calabi-Yau threefolds. Our method has a combinatorial representation by the means of double octic diagrams. The proposed algorithm is applicable both in classical context over a complex disk as well as in arithmetic setting over a spectrum of DVR. We illustrate algorithm's efficacy through three examples where we compute semistable degeneration and limiting mixed Hodge structure for explicit families of double octics.

\end{abstract}

\section{introduction}

The main goal of the present article is to develop an efficient method of constructing semistable reduction of a one-parameter family of Calabi-Yau threefolds of double octic type. A Calabi-Yau threefold is a smooth, complex, projective manifold of dimension 3 with trivial canonical bundle and vanishing first Betti number. Calabi-Yau threefolds are three dimensional analogons of elliptic curves and K3 surfaces, contrary to the lower dimensional cases Calabi-Yau threefolds are not all homeomorphic. It is even not clear that they live in a finite number of families or that the numerical invariants (the Euler characteristic, Hodge numbers) are bounded.

Semistable degenerations and their associated monodromy weight spectral sequence are powerful tools to study algebraic schemes in both analytic and arithmetic setting (cf. \cite{KN}, \cite{DT}, \cite{DHT}, \cite{Lee}).

Semistable degenerations of families of elliptic curves are well-understood, in fact singular fibers of any relatively minimal family of elliptic curves are classified. Any semistable degeneration of K3 surfaces has a modification which is a Kulikov model (provided all components of the central fiber are k\"ahler). There are three types of central fibers of a Kulikov model (cf. \cite{Kulikov, PP}) depending on the unipotency index of the local monodromy. A similar phenomenon dose not hold for degenerations of Calabi-Yau threefolds (cf. \cite{CvS}) where monodromies do not classify types of central fiber. One of the main obstacles for developing a relevant theory is a lack of examples.

Computation of a semistable degeneration often requires intricate geometric reasoning and complex cohomological computations that make the straightforward approach infeasible. In this paper we present an explicit algorithm for computing semistable degeneration of one-parameter families of \emph{double octic Calabi-Yau threefolds}. A double octic is a crepant resolution of a double cover of the projective space  $\mathbb P^3$ branched along an \emph{octic arrangement} which is an arrangement of eight planes satisfying some very mild generality conditions. Double octic Calabi-Yau threefolds are very suitable for explicit computations as they admit an explicit geometric (or even combinatorial) description, while on the other hand they are rich enough to provide examples of many geometric and arithmetic phenomena.

One-parameter families of double octics were classified in \cite{meyermodular}, there are exactly 64 one-parameter families of double covers of $\PP^3$ branched along an octic arrangement $\mathcal X\longrightarrow \PP^1\setminus \Sigma$ all of them over a complement of a finite set $\Sigma$ in the projective line $\mathbb P^1$. A degeneration of  $\mathcal X$ at a point $w\in \Sigma$ is given by a complement of $\mathcal X$, it is not uniquely determined by the family $\mathcal X$, but it depends on an explicit parametrization. We shall consider only degenerations, when $X_w$ is also defined by a double cover of $\mathbb P^3$ branched along an octic arrangement. In this situation, the octic arrangement at $w$ has different singularities than other fibers of $\mc{X}$. As a consequence, a fiberwise resolution of singularities of generic fibers produces only a partial resolution of singularities of the degenerate fiber $X_w$.

To keep track of singularities of the degenerate fiber at each step of a resolution, we introduce octic diagrams (sect. \ref{diagrams}). From the combinatorial analysis of octic diagrams we are able to compute singularities of central fibers in each of eleven types of degenerations (described in Appendix \ref{singularity classification}). As a consequence we prove that there are four types of singularities in the central fiber (sect. \ref{diagrams}). From this local description of a partial resolution we derive singularities of the degenerate fiber for every degeneration listed in Appendix \ref{full list of degenerations}. Then the semistable reduction can be constructed by blowing up singular locus and taking finite coverings. Important cohomological information relating the central and generic fibers of a semistable degeneration are contained in the Clemens-Schmidt and the monodromy weight spectral sequences. We compute the former for three sample degenerations (which were chosen so that they demonstrate all possible types of local singularities). From the monodromy weight filtration we derive limiting Hodge structures and mixed Hodge structures on the degeneration, however, we refrain from delving here into that topic, as well as other applications,  to keep the paper concise.

Methods presented here easily translate to the arithmetic setting, where varieties are defined over a spectrum of discrete valuation rings. As an example in \cite{CO} we used result of this paper adapted to the arithmetic setting to prove that the Galois representation on $H^3_\text{\'et}(X,\mathbb{Q}_l)$ is unramified for $l\not=p$ and crystalline for $l=p$, for a certain double octic Calabi-Yau threefold $X$ over $\mathbb{Q}_p$, which provided a counterexample to the N\'eron-Ogg-Shafarevich criterion in dimension three. These techniques could also be used to give much simpler proofs for problems like those in \cite{ex} where one computes Betti numbers of schemes in positive characteristic.

This paper is structured as follows: we begin with the preliminaries, outlining the necessary background and definitions. We then move on to the detailed algorithm for resolving double octics, followed by a discussion on the infinitesimal deformations of double octics and their universal deformation families. The subsequent sections cover combinatorial methods for obtaining semistable reduction and the computation of the monodromy weight spectral sequence. Throughout the paper, we use diagrams to illustrate the degenerations of double octics and provide a comprehensive analysis of the singularities and their resolutions.

\section{Preliminaries}\label{the algorithm}

By a double cover of $V$ branched along $B$ we mean a normal algebraic variety $X$ together with a finite morphism of degree two $\pi \colon X \rr V$ such that the fiber $\pi^{-1}(v)$ is a single point for $v \in B$ and two points for $v \notin B$. When $V$ is smooth there exists a line bundle $\mc{L}$ on $V$ such that $\pi_* \mc{O}_X \cong \mc{O}_V \otimes \mc{L}^{-1}$. Then $\mc{O}_V(B) \cong \mc{L}^{\otimes 2}$. On the other hand if $\mc{L}$ is a line bundle on an algebraic variety $V$ such that $\mc{L}^{\otimes 2}$ is isomorphic to $\mc{O}_V(B)$ for a reduced divisor $B$ then there is a double cover $\pi \colon X \rr V$ of $V$ branched along $B$. In particular double cover is not uniquely determined by the branch locus alone. In this paper we consider only the case of double covers of algebraic varieties without 2-torsions in the Piccard group. In that case the double cover is uniquely determined by the branch locus.

In this paper, we are interested in a special type of double covers, namely double covers of $\PP^3$ branched along a union of eight distinct planes such that there are no $q$-fold curves for $q \geq 4$ and no $p$-fold points for $p \geq 6$. Double covers of $\PP^3$ branched along such planes arrangements were introduced in \cite{CYB}, we refer to their branching locus as \emph{octic arrangements}. The main motivation to consider them is the following theorem.

\begin{theorem}[\cite{CYB}]
\label{resolution of octic}
Let $D \subset \PP^3$ be an octic arrangement. Then there exists a sequence $\sigma = \sigma_1 \circ ... \circ \sigma_s \colon \wdt{\PP^3} \rr \PP^3$ of blowing-ups and a smooth and even divisor $D^* \subset \wdt{\PP^3}$ such that $\sigma_*(D^*) = D$ and the double covering $Y$ of $\wdt{\PP^3}$ branched along $D^*$ is a smooth Calabi-Yau threefold.
\end{theorem}

Smooth Calabi-Yau threefolds obtained this way are called \emph{double octics}. An explicit algorithm for resolving a double cover of an octic arrangement is the following.

Given an octic arrangement $D \subset \mathbb{P}^{3}$, and a double cover $X \rightarrow \mathbb{P}^{3}$ branched along $D$, we perform the following steps.

Step 1. Blow up each fivefold point of the arrangement to get $\sigma_{1}: T_{1} \rightarrow \mathbb{P}^{3}$. Let $D_1$ be a sum of a strict transform of $D$ plus the exceptional divisor. 


Step 2. Blow up all triple lines $\sigma_{2}: T_{2} \rightarrow T_{1}$. Note that, after step 1. all triple lines are pairwise disjoint. Let $D_2$ be the sum of a strict transform of $D_1$ plus the exceptional divisor. 


Step 3. Blow up each quadruple point, $\sigma_{3}: T_{3} \rightarrow T_{2}$. Let $D_3$ be a strict transform of $D_2$. 


Step 4. Blow up each double line, $\sigma_{4}: T_{4} \rightarrow T_{3}$. Let $D_4$ be a strict transform of $D_3$. We define $Y$ as a double cover of $T_{4}$ branched along $D_4$. The result of this step depends on the order in which we blow up lines.

The threefold $Y$ is a smooth Calabi-Yau manifold, as its branching locus of is a disjoint union of smooth surfaces, and each step was a crepant blow-up.

Throughout the paper we use the following naming convention. 
\begin{itemize}
    \item $\mc{D}$ is a family of planes arrangements.
    \item $\mc{X}$ is a family of singular double covers branched along members of $\mc{D}$. 
    \item $\mc{Y}$ is a family whose generic fiber are smooth double octics.
\end{itemize}

The capital letters $D,X,Y$ stand for members of respective families.

\section{Infinitesimal deformations of double octics}
\label{extending resolution}

By the Bogomolov-Tian-Todorov theorem, for each Calabi-Yau threefold $Y$ with $h^{1,2}(Y)=1$, there exists a smooth family $\mathcal{Y}$ over a complex disk, which is universal for $Y$. In case of a double octic, we can construct a universal family explicitly as a quasi-projective variety, such families are completely classified in \cite{CYK}.

We can construct universal family as follows. For every double cover $X \rr \PP^3$ branched along 8 planes, with a resolution $Y$ having $h^{1,2}(Y)=1$, we associate a family $\mc{D}$ of planes arrangements, whose member is the branching locus of $X$.

$$
\mc{D} := \{((x:y:z:t),(A:B)) \in \PP^3 \times (\PP^1 \setminus \Sigma) :P(x,y,z,t,A,B)=0\},
$$
The projective variable $(A: B) \in \PP^1$ is the parameter of the family, $\Sigma \subset \mathbb{P}^{1}$ is a finite set and $P(x, y, z, t, A, B)=$ $\prod_{i=1}^{8} P_{i}(x, y, z, t, A, B)$ is a product of eight polynomials $P_{i}$ linear in variables $x, y, z, t$ and homogeneous in $(A: B)$. Moreover, we require that planes in all members of such family have the same incidence relations between them, as the branching locus of $X$. We refer to arrangements satisfying this property as having the same \emph{incidences}. A family of double covers $\mc{X}$ over $\mathbb{P}^{1} \backslash \Sigma$ can be extended to a family $\wb{\mc{X}}$ over the entire $\PP^1$ by taking a closure of double covers of $\mathbb{P}^{3}$ branched along the closure of  $\overline{\mathcal{D}}$ of $\mc{D}$. The fibers of $\wb{\mc{X}}$ which have different incidences are called \emph{degenerate} fibers, while the other ones are called \emph{generic} fibers. The particular closure $\overline{\mc{D}}$ is determined by the equations for $\mc{D}$ we choose.

By the Theorem \ref{resolution of octic} each fiber of $\mathcal{X}$ whose branching divisor is an octic arrangement can be resolved to a smooth Calabi-Yau threefold. We show that one can obtain a map $\sigma: \mathcal{Y} \rightarrow \mathcal{X}$, where $\mathcal{Y}$ is a family whose generic fibers are smooth Calabi-Yau threefolds, and the fiberwise restriction $\sigma_w \colon Y_w \rr X_w$ for any $w \in \PP^1 \setminus \Sigma$, is the resolution from the Theorem \ref{resolution of octic}. We make two observations.

\begin{observation}
\label{centers are intersections}
Any center of the blow-up that appears in a resolution of a double cover of an octic arrangement can be defined as an intersection of a subset of components in the branching divisor.
\end{observation}

\begin{observation}
\label{intersections can be extended}
Let $\mc{D}$ be as above. For any $w \in \PP^1\setminus \Sigma$ and any intersection $Z$ of a subset of planes in $D_w$, there exists an intersection $\mc{Z}$ of hypersurfaces in $\mc{D}$, such that $\mc{Z}_w = Z$.
\end{observation}

These observations imply that for a generic fiber $X$ of a family $\mc{X}$ we can extend its resolution to the whole family. Other generic fibers are resolved as well, as the blow-ups are determined by combinatorial incidences in the branching locus. Degenerate fibers, on the other hand, might be singular. It's worth adding that changing the order in which we blow up double lines, not only affects generic fibers, but may also change the degenerate fiber.

We finish this section by describing how blow-ups of generic fibers affect degenerate fibers. We remind that each step in the resolution from section $\ref{the algorithm}$ consists of taking a blow-up $\sigma$ and defining a new branching divisor $D^* := \sigma^* D - \sum_i 
 2 \lfloor \frac{a_i}{2} \rfloor E_i$, where $E_i$ are components of the exceptional divisor of the blow-up $\sigma$ and $a_i$ are multiplicities of centers of blow-ups corresponding to $E_i$.

\begin{lemma}
\label{commutativity of resolution}
Let $\mc{X}$ be a family of double covers branched along a family of octic arrangements $\mc{D}$. Let $X_0$ be a fiber of a family $\mc{X}$ and let $\sigma: \widetilde{\mathcal{X}} \rightarrow \mathcal{X}$ be a blow-up with an exceptional divisor $\mathcal{E}$ and a center $\mathcal{Z}$ such that no component of $\mc{Z}$ is inside $X_{0}$. Then, $\left.\sigma\right|_{X_{0}}: \sigma^{-1}\left(X_{0}\right) \rightarrow X_{0}$ is a blow-up with the center $\mathcal{Z} \cap X_{0}$. If $\mathcal{D}$ is a divisor in $\mathcal{X}$, which has no component contained in $X_{0}$ and $\mathcal{D}^{*}:=\sigma^{*}(\mathcal{D})-k \mathcal{E}$ for some $k \in \ZZ$, then $\left.\mathcal{D}^{*}\right|_{X_{0}}=\sigma^{*}\left(\left.\mathcal{D}\right|_{X_{0}}\right)-k E$, where $E=\left.\mathcal{E}\right|_{X_{0}}$ is an exceptional divisor of $\sigma|_{X_{0}}$.
\end{lemma}
\begin{proof}
The blow-up $\sigma$ is given by a morphism $\phi: \mathcal{X} \backslash \mathcal{Z} \rightarrow \mathbb{P}^{n}$ which has the property that $\widetilde{\mathcal{X}}=\overline{\operatorname{graph} \phi} \subset \mathcal{X} \times \mathbb{P}^{n}$, thus $\sigma^{-1}\left(X_{0}\right)=\left.\overline{\operatorname{graph} \phi}\right|_{X_{0}}$. Since $\overline{\mathcal{Z} \backslash X_{0}}=\mathcal{Z}$, we have that $\left.\overline{\operatorname{graph} \phi}\right|_{X_{0}}=\overline{\left.\operatorname{graph} \phi\right|_{X_{0}}}$, which is by the definition the blow-up of $X_{0}$ in $\mathcal{Z} \cap X_{0}$.

For the second part we recall that pullbacks commute with restrictions, thus $\left.\sigma^{*}(\mathcal{D})\right|_{X_{0}}=$ $\left.\sigma^{*}\right|_{X_{0}}(D)$, it follows that $\left.\left(\sigma^{*}(D)-k \mathcal{E}\right)\right|_{X_{0}}=\left.\sigma^{*}\right|_{X_{0}}(\mathcal{D})-k E$.
\end{proof}

We shall mostly use the following corollary of this lemma which applies when the multiplicity of a subvariety in the branching locus rises in degenerate fibers.

\begin{corollary}
\label{jumping multiplicity}
    Let $\mc{X}, \mc{D}$ and $\mc{Z}$ be as in the lemma above. If there is a component $\mc{D}_i$ of $\mc{D}$ and a fiber $X_0$ of $\mc{X}$ such that $\mc{Z} \cap X_0 \subset \mc{D}_i$ and $\mc{Z} \not\subset \mc{D}_i$, then the blow-up $\wdt{\mc{X}}$ of $\mc{X}$ along $\mc{Z}$ introduces a new component in the branching locus of $\wdt{X_0}$. Moreover this extra component lies in a strict transform of $\mc{D}_i$.
\end{corollary}

\section{Octics diagrams}
\label{diagrams}
Throughout the rest of this article we shall use certain diagrams to study degenerations of double octics. By a degeneration we mean a proper map $\mc{X} \rr \Delta$, such that the map $\mc{X} \times_{\Delta} \Delta^* \rr \Delta^*$ is smooth. The purpose of these diagrams is to present the geometry of the central fiber \emph{at each step} of a resolution of generic fibers. Diagrams consist of squares, each representing a component of the branch divisor of the central fiber at some step of a resolution. There are three types of components in the branch divisor of the central fiber that appear during a resolution.

\begin{enumerate}
  \item Strict transforms of planes from the octic arrangement we started with. These are denoted by $P_1,P_2, \ldots$ following the order of linear forms in the equation,

  \item Exceptional divisors that are added to the branch divisor of all fibers after blowing up a five-fold point or a triple line in a generic fiber. These are labeled by capital letters $A, B, C, \ldots$,

  \item It may happen that we blow up a double line or a quadruple point in generic fibers, which degenerates to a triple line or to a quintuple point in a special fiber respectively. From the cor. \ref{jumping multiplicity} it follows that after such blow-up we have a new component in the branching locus of the central fiber. We label it by adding prime to the component in which strict transform of the new surface lies.
\end{enumerate}

We use the following convention to visualize further properties of the branch locus.
\begin{enumerate}
    \item For each intersection of surfaces $I$ and $J$ along a curve we add a line labeled $i$ to the square $J$, and a line labeled $j$ to the square $I$. In case two surfaces $J$ and $K$ intersect the surface $I$ along the same line, we label the corresponding line in $I$ as $jk$.

    \item To denote that a certain curve is an exceptional divisor of a surface we draw it as a dashed line.

    \item Lines on diagrams intersect if and only if their corresponding curves intersect as well. To denote that two curves do not intersect, we either draw them as disjoint lines or replace a corresponding point on a diagram by a circle.
\end{enumerate}

Throughout the the rest of the paper we use the following convention for naming singularities. We label a multiple line as $l_i$, where $i$ is the multiplicity of the line and we label multiple point as $p^j_i$, where $i$ is the multiplicity of the point, and $j$ is the number of triple lines on which it lies. Furthermore the family that we obtain from $\mc{X}$ by resolving all generic fibers, with degenerate fibers being possibly singular, is denoted as $\mc{Y}$ and the central fiber of $\mc{Y}$ is denoted $Y_0$.

In the remainder of this section we compute the singularities of degenerate fibers for all degenerations listed in the appendix \ref{singularity classification}. We start with a description of a resolution of singularities of a generic fiber as a sequence of blowing-ups and then extend it to the other generic fibers and also the central fiber by prop. \ref{extending resolution}. Singular curves in generic fibers are denoted as $L_{i_1,i_2,...}$ , where $i_1,i_2,...$ are the surfaces on which the singularity lies. For each blow-up we draw a series of diagrams which represents incidences between surfaces in the branching locus at each step. We finish each subsection with a description of singularities in the central fiber after all singularities in generic fibers are resolved.


\subsection{New $l_3$ line}
We start with the equation 
$$ 0 =xy(x+y+w).$$
Resolution of a generic fiber consists of blowing up double lines $L_{12}, L_{23}, L_{13}$. The central fiber has the equation: $0=x y(x+y)$, it can be presented as the following picture.

\includegraphics[scale=0.8]{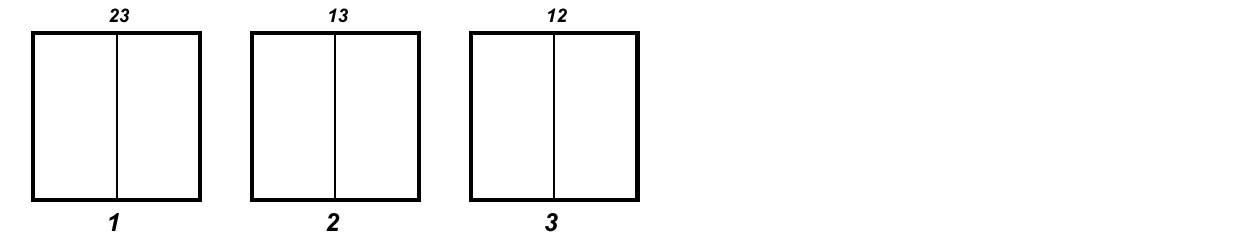}

In a generic fiber, we start by blowing up the line $L_{12}$ and then we remove the exceptional divisor lying above the two surfaces $P_1,P_2$ from the pullback of the branch locus. However, in the central fiber, the line $L_{12}$ lies also on the plane $P_3$, so the center of this blow-up becomes a triple curve, this introduces a new component to the branch locus.

\includegraphics[scale=0.8]{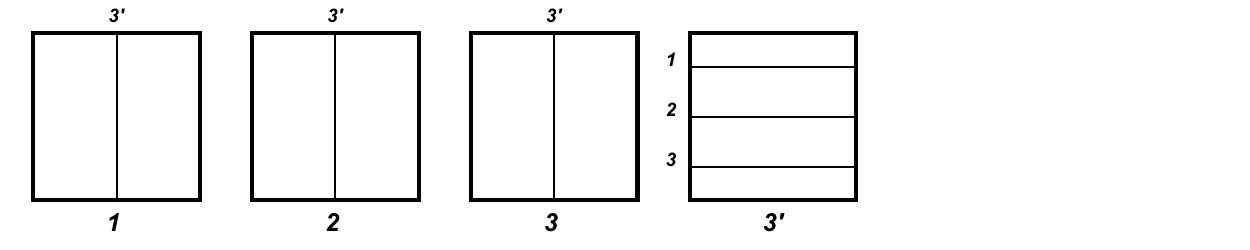}

In the above four diagrams we denote by $P_3'$ and $P_3$ the two components into which the plane 3 splits. The other two blow-ups in generic fibers are the blow-ups of lines $L_{13}, L_{23}$. In the central fiber, they are represented as $P_1 \cap\left(P_3 \cup P_3'\right)$ and $P_2 \cap\left(P_3 \cup P_3'\right)$ The diagram after these blow-ups is the following.

\includegraphics[scale=0.8]{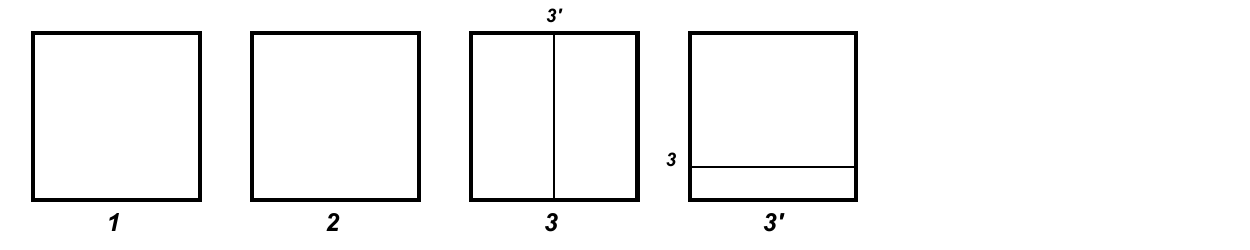}

After resolving generic fibers, we get a single double line in the central fiber lying over the triple line from the original arrangement.


\subsection{New $p_{4}^{0}$ point}\label{p04 point}
We start with the equation
$$0=xyz(x+y+z+w).$$
Resolution of a generic fiber consists of blowing up the lines $L_{12}, L_{13}, L_{23}, L_{14}, L_{24}, L_{34}$. In the central fiber, the equation becomes $0=x y z(x+y+z)$. This can be visualized as follows.

\includegraphics[scale=0.8]{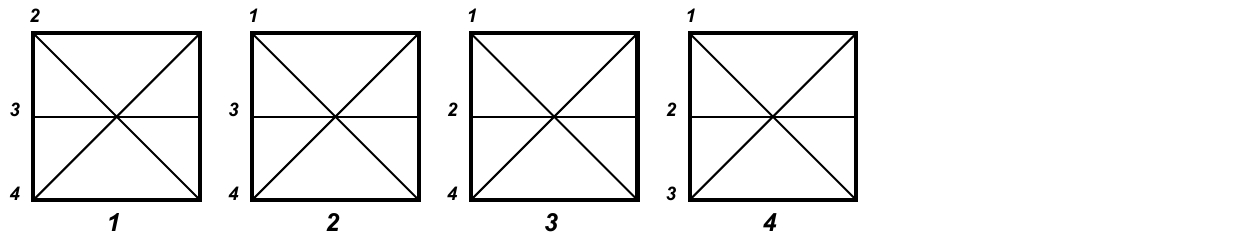}

The blow-up of the line $L_{12}$ resolves double lines on surfaces $P_1,P_2$, while surfaces $P_3,P_4$ are blown up at one common point. This means that after the blow-up the surfaces $P_3,P_4$ intersect along a line $L_{34}$ and their common exceptional divisor.

\includegraphics[scale=0.8]{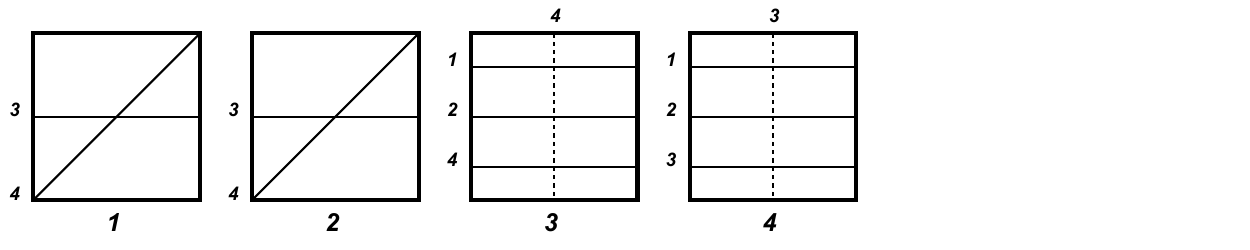}

The next blow-up of the line $L_{13}$ separates surfaces $P_1,P_3$ and blows up the surface $P_4$ at a point that lies on its exceptional divisor. This can be presented as follows.

\includegraphics[scale=0.8]{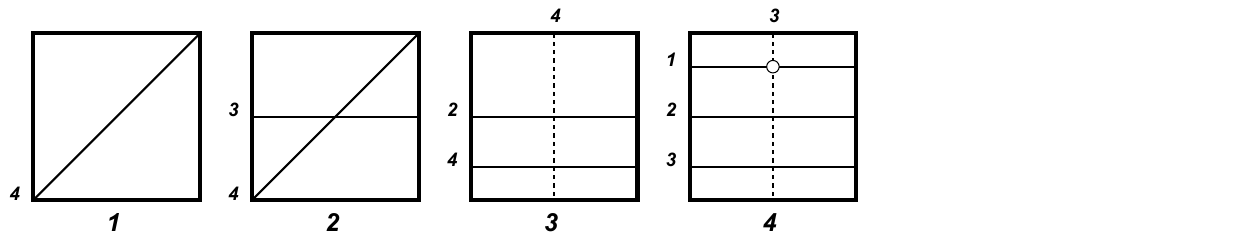}

Next, we blow the curve $L_{23}$.

\includegraphics[scale=0.8]{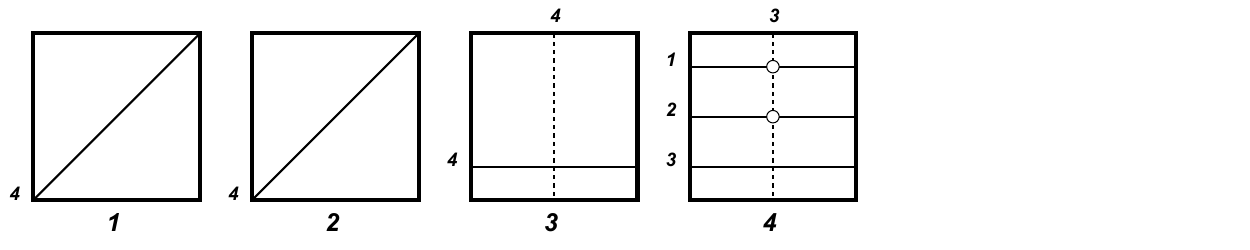}

The last blow-ups, first of the curve $L_{14}$ and followed by the blow-up of $L_{24}$, resolve the corresponding double curves on surfaces $P_1,P_2,P_4$.

\includegraphics[scale=0.8]{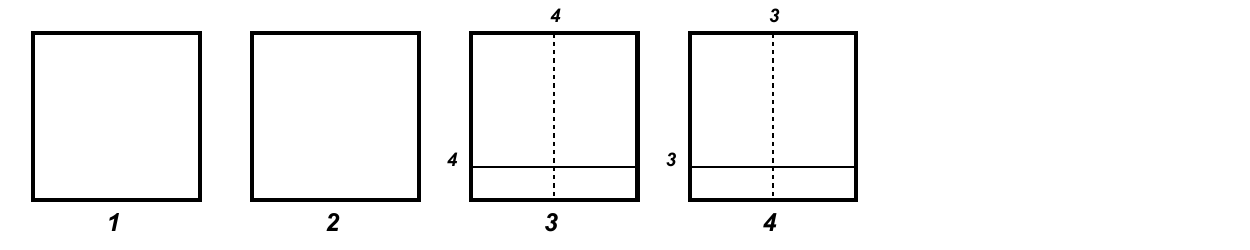}

Finally, we blow up the curve $L_{34}$. The surfaces $P_3, P_4$ intersect along two curves, one of them is the strict transform of the double line that comes from a generic fiber, and the other one is the exceptional divisor from the first blow-up. These two curves intersect at a single point. The center of this blow-up is a union of two transversely intersecting curves, thus the blow-up introduces a node in the ambient space outside the branching locus. The resulting double cover will have a pair of nodes. We remark that the pullback of any of the two blown-up lines is a smooth surface isomorphic to $\PP^1 \times \PP^1$ which contains these nodes, this means that these nodes have a small, projective resolution.


\subsection{A $p_{5}^{1}$ point degenerates to a $p_{5}^{2}$ point}

We start with the equation
$$0=xy(x + y)z(x + wy + z).$$

Resolution of a generic fiber consists of three steps: blowing-up the fivefold point, the triple line $L_{123}$ and finally all double lines. The central fiber can be visualized as follows.

\includegraphics[scale=0.8]{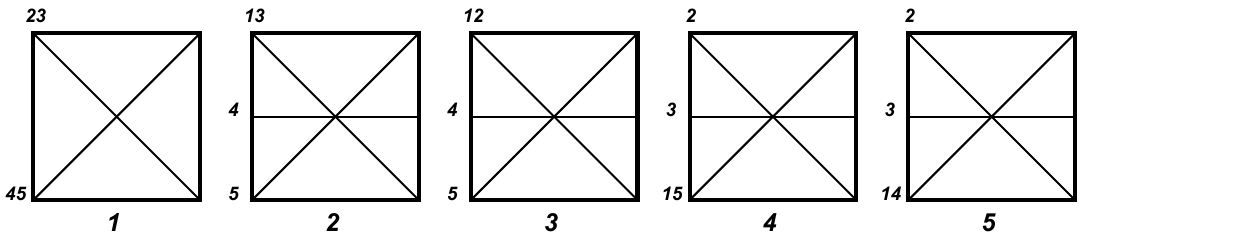}

The blow-up of the fivefold point separates all lines passing through it and introduces a new surface $P_A$ in the branching divisor.

\includegraphics[scale=0.64]{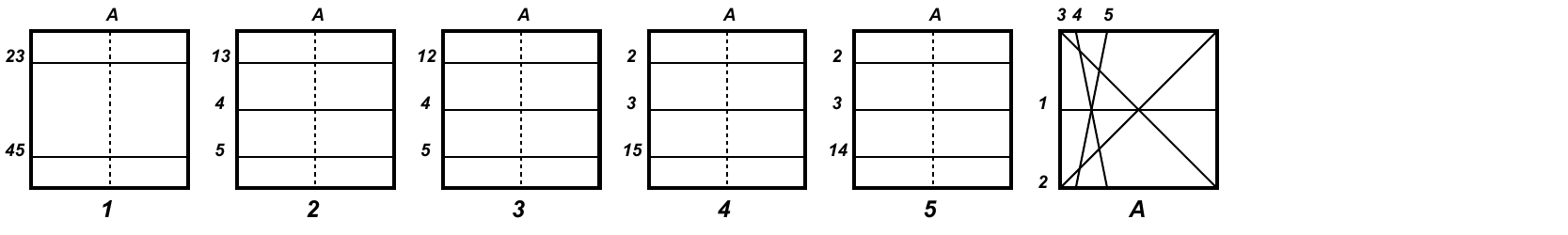}

We blow up the triple line $L_{123}$. As a result of that, a new component $P_B$ appears in the branching divisor and the surface $P_A$ is blown up at a point.

\includegraphics[scale=0.64]{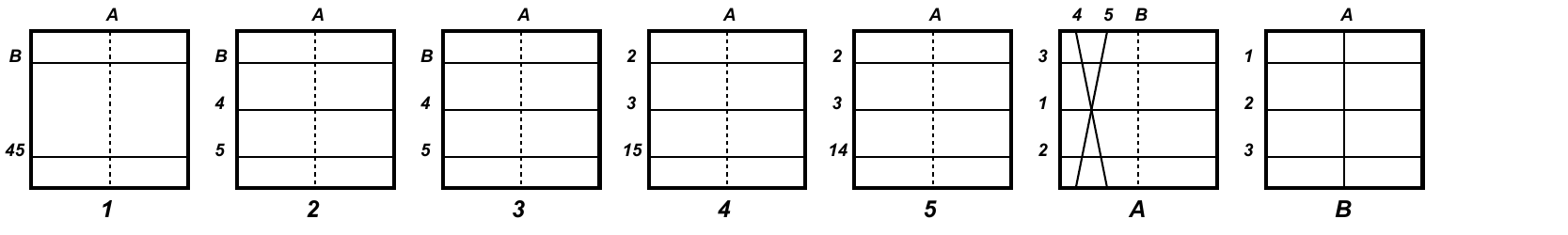}

We blow up the remaining double lines from a generic fiber: $L_{1 B}, L_{2 B}, L_{3 B}, L_{24}$, $L_{25}, L_{34}, L_{35},$ $L_{A B}, L_{1 A} L_{2 A}, L_{3 A}, L_{45}, L_{14}, L_{15}, L_{4 A}, L_{5 A}$. We start with the lines: $L_{1 B}$, $L_{2 B}, L_{3 B}, L_{24},$ $L_{25}, L_{34}, L_{35}$, they are pairwise disjoint, so we can blow them up in a single step.

\includegraphics[scale=0.64]{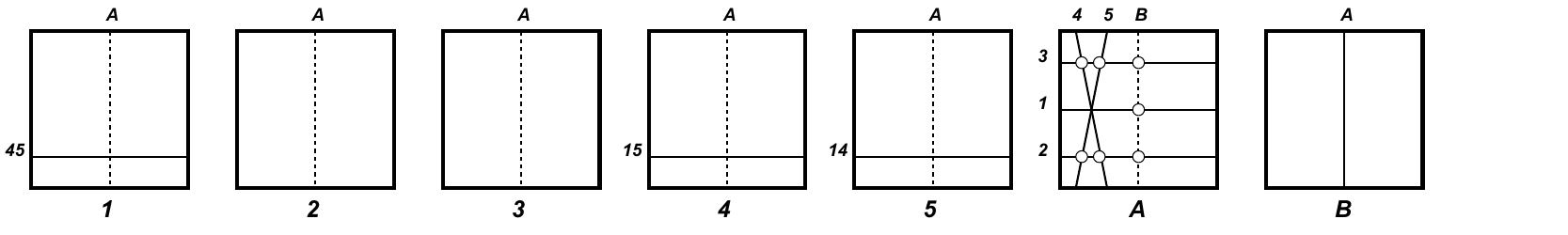}

We blow up the lines: $L_{2 A}, L_{3 A}, L_{A B}$. After that, surfaces $P_2,P_3,P_B$ are disjoint from other surfaces, so we no longer include them in our diagrams. We are left with the following configuration.

\includegraphics[scale=0.64]{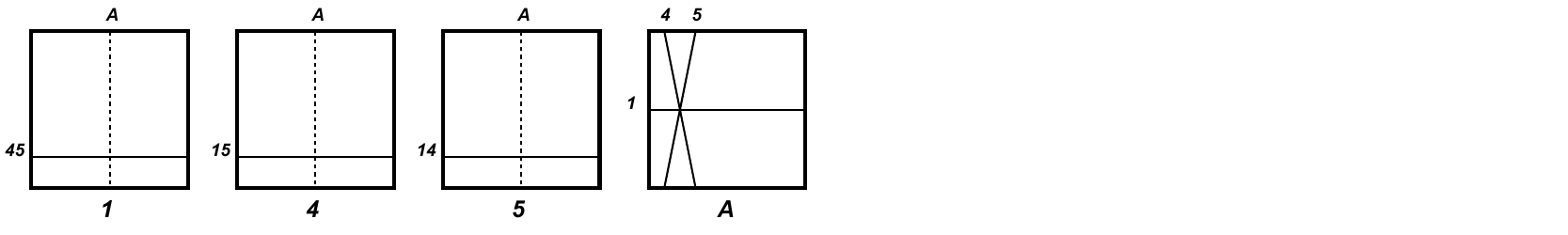}

We blow up the line $L_{14}$. By the same arguments as in the degeneration to the triple line, we get the following diagram.

\includegraphics[scale=0.64]{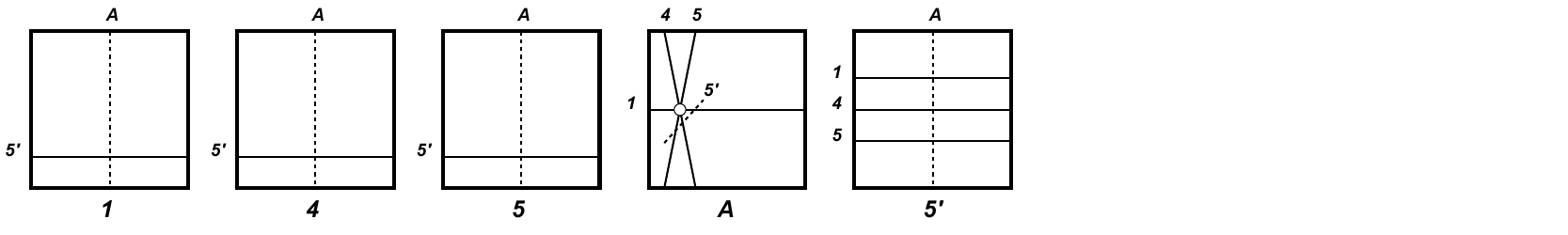}

We blow up the curves $L_{A 1}, L_{A 4}$.

\includegraphics[scale=0.64]{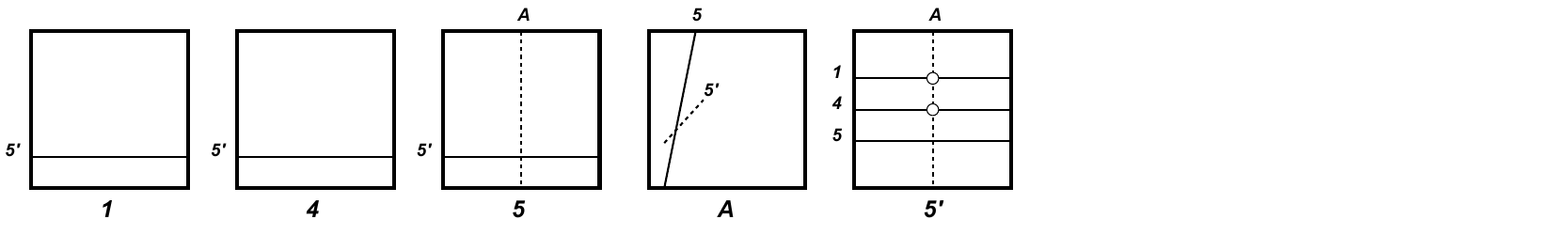}

Blow-ups of curves $L_{15}, L_{45}$ separate surfaces 1 and 4 from the others.

\includegraphics[scale=0.64]{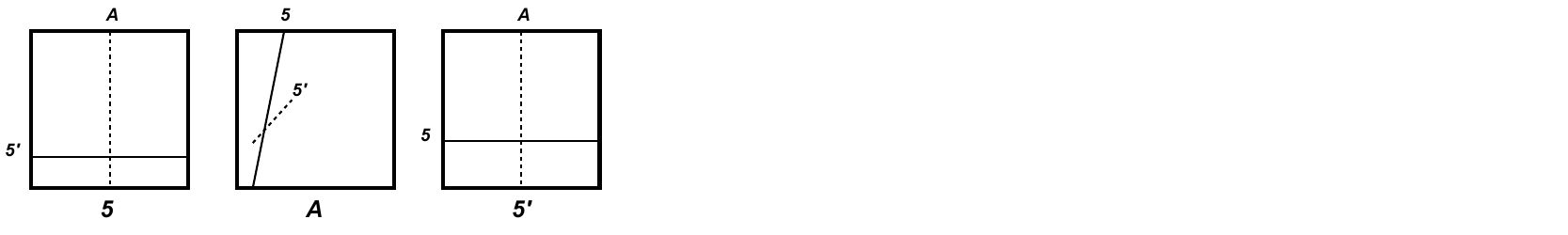}

The blow-up of the curve $L_{A 5}$ is a blow-up of two intersecting curves, where the point of intersection lies on the curve where surfaces $P_5$ and $P_5'$ intersect. As we performed all blow-ups in generic fibers the double curve where surfaces $P_5$ and $P_5'$ intersect will remain in the central fiber.

Concluding, after resolution of generic fibers, there is one double curve in the central fiber with a single pinch point on it.


\subsection{Two $p_{4}^{1}$ points collide to a $p_{5}^{2}$ point}

We start with the following equation

$$x y (x+y) z (x+z+w)=0.$$

Generic fibers do not have a fivefold point, but only a triple curve $L_{123}$ and two fourfold points of type $p^1_4$ that lie on it. Thus, a resolution of a generic fiber starts with a blow-up of the triple curve, which resolves fourfold points and adds a new surface $P_A$ to the branching divisor. After that we blow up all remaining double lines.

The starting arrangement of planes in the central fiber looks as follows.

\includegraphics[scale=0.64]{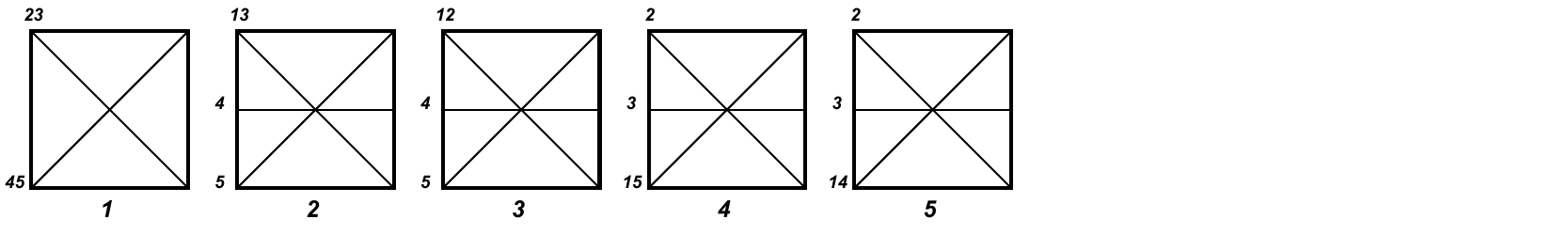}

Blow up the triple line $L_{123}$.

\includegraphics[scale=0.64]{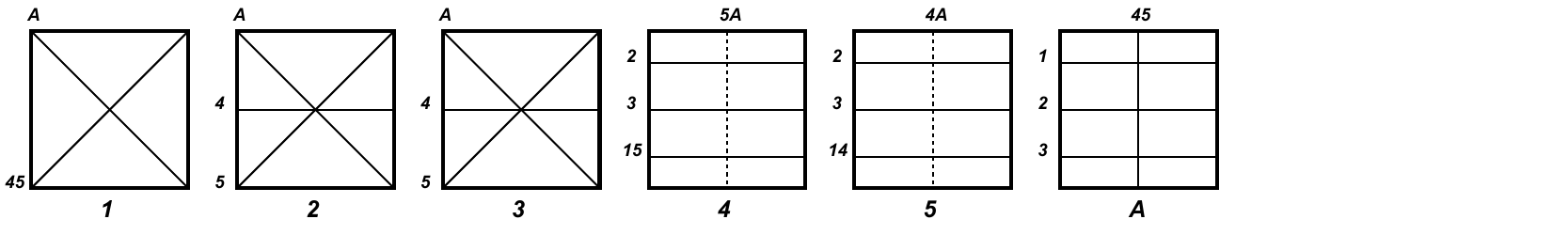}

Blow up curves $L_{1A}, L_{2A}, L_{3A}$.

\includegraphics[scale=0.64]{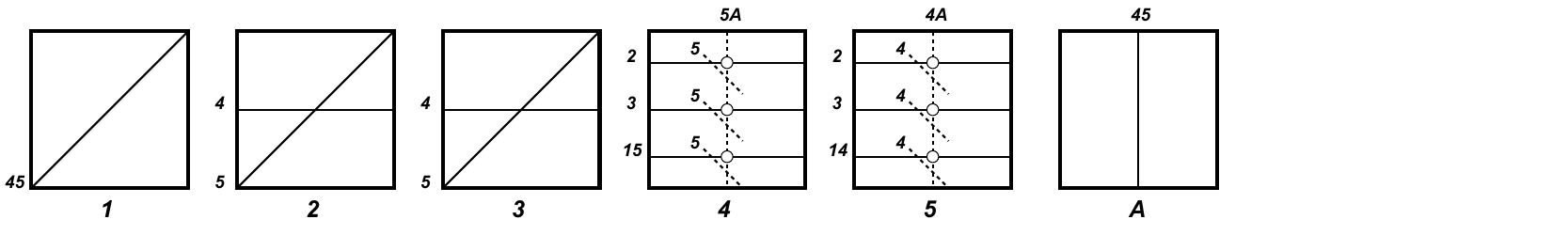}

Blow up curves $L_{24}, L_{25}, L_{34}, L_{35}$.

\includegraphics[scale=0.64]{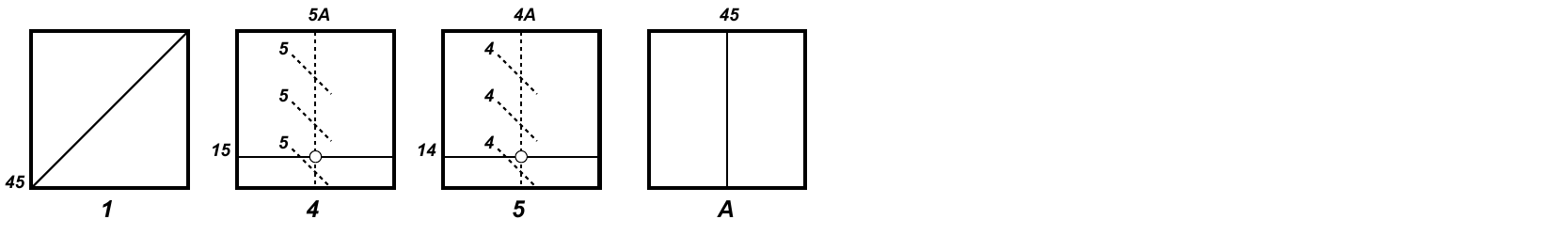}

Blow up curve $L_{14}$.

\includegraphics[scale=0.64]{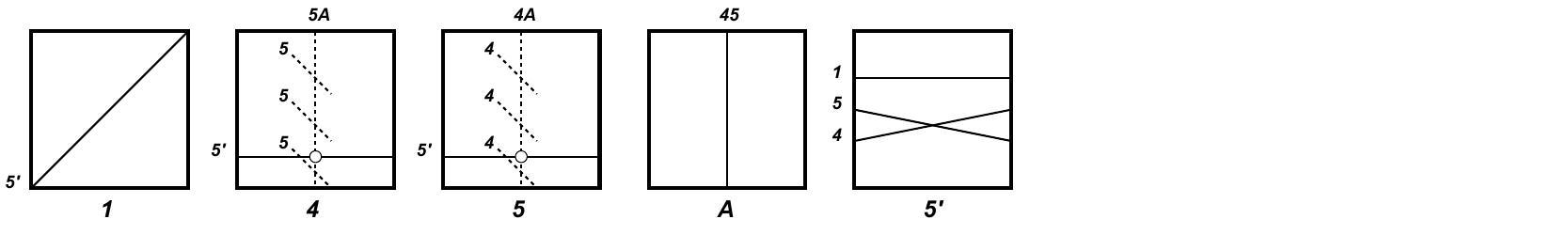}

Blow up curve $L_{5A}$.

\includegraphics[scale=0.64]{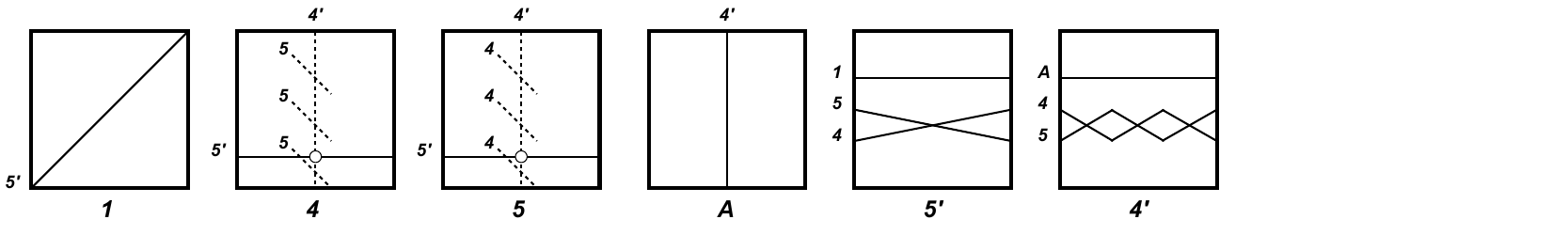}

Blow up curves $L_{4A},L_{15}$.

\includegraphics[scale=0.64]{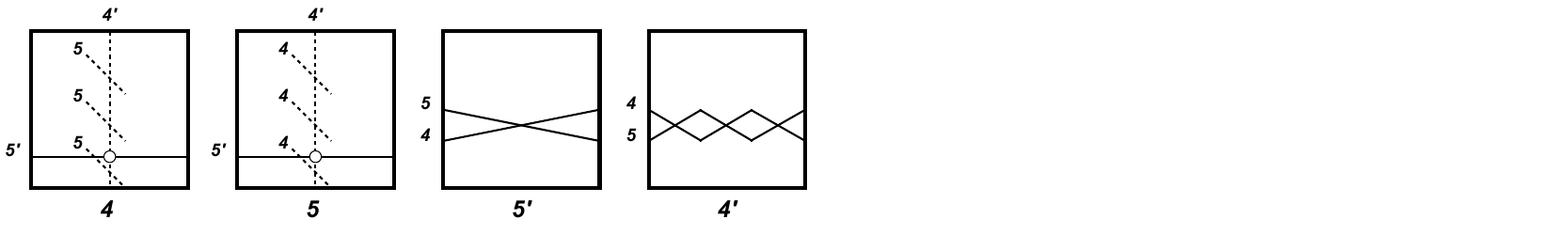}

The last blow-up of the curve $L_{45}$ introduces 3 pinch points on the curve where surfaces $P_4$ and $P_4'$ intersect and a single pinch point on the curve where surfaces $P_5$ and $P_5'$ intersect

Concluding, after resolution of generic fibers in the degenerate fiber there are two disjoint double curves over the triple line $L_{145}$. One of these curves has 3 pinch points and the other one has a a single one.


\subsection{Two $p^1_4$ points collide to a $p^1_5$ point}

We start with the following equation.

$$xy(x + y)z(x + y + z + w) =0$$

Similarly as in previous degeneration, generic fibers do not have a fivefold point, but only a triple curve. The resolution of generic fibers starts with a blow-up of a triple curve, after which we blow up double curves in generic fibers.

The starting arrangement of planes in the central fiber looks as follows.

\includegraphics[scale=0.64]{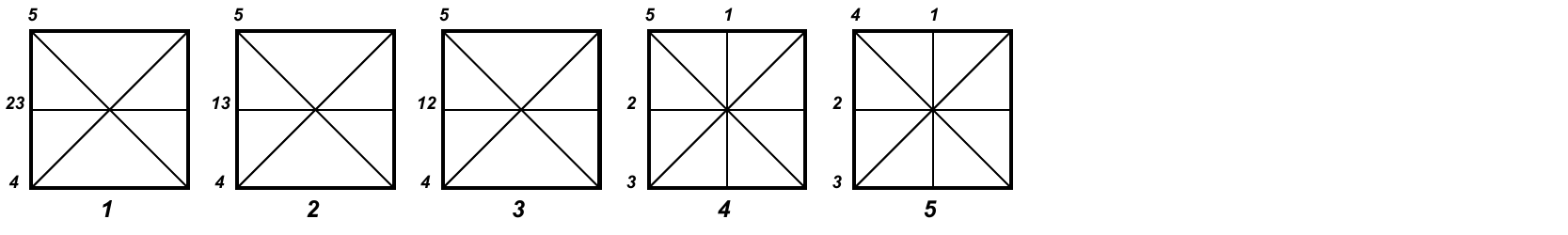}

The first blow-up of the triple line $L_{123}$ has the following effect.

\includegraphics[scale=0.64]{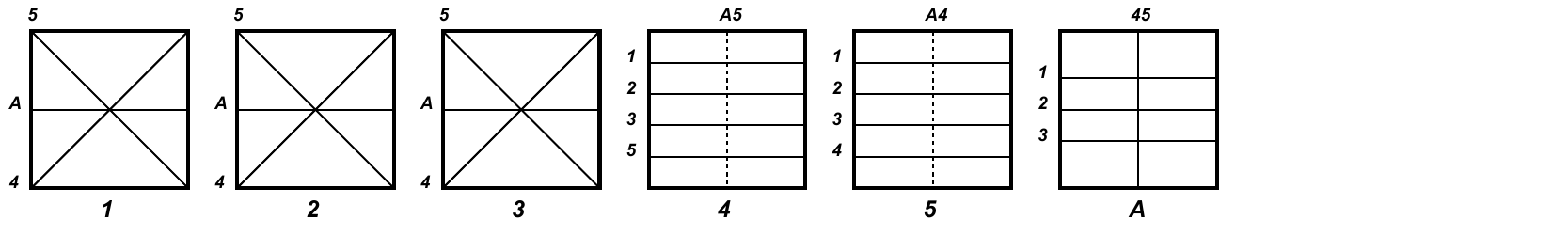}

Now, we blow up the double lines: $L_{14}, L_{15}, L_{1A}, L_{24}, L_{25}, L_{2A}, L_{34}, L_{35}, L_{3A}$ (in this order). After that the surfaces $P_1,P_2,P_3$ is separate from the others, so we remove them from diagrams.

\includegraphics[scale=0.64]{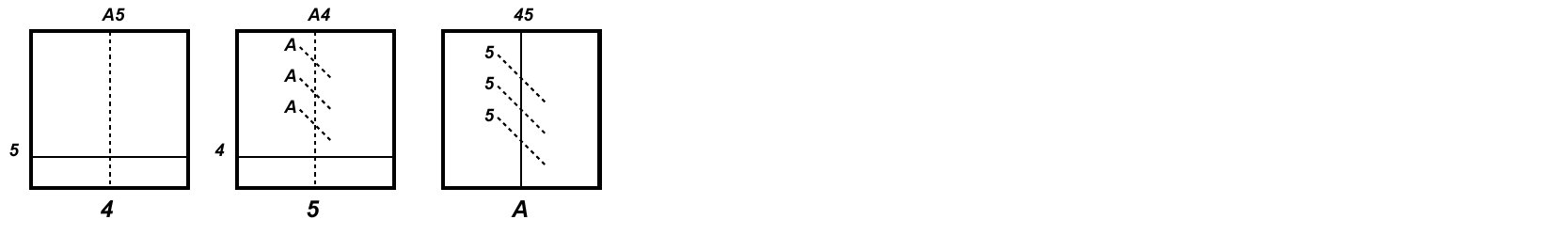}

The only remaining double lines are: $L_{45}, L_{4A}, L_{5A}$. We first blow up the curve $L_{4A}$. 

\includegraphics[scale=0.64]{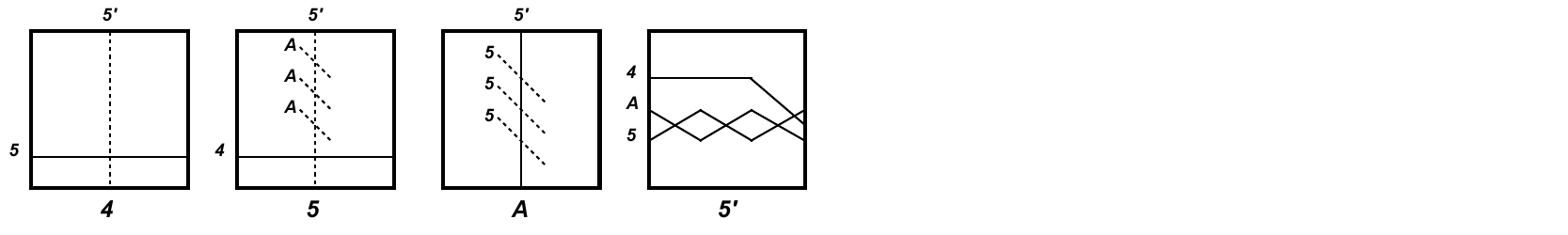}

Next, we blow up the line $L_{45}$, this creates a pinch point on the curve where surfaces $P_5$ and $P_5'$ intersect.

\includegraphics[scale=0.64]{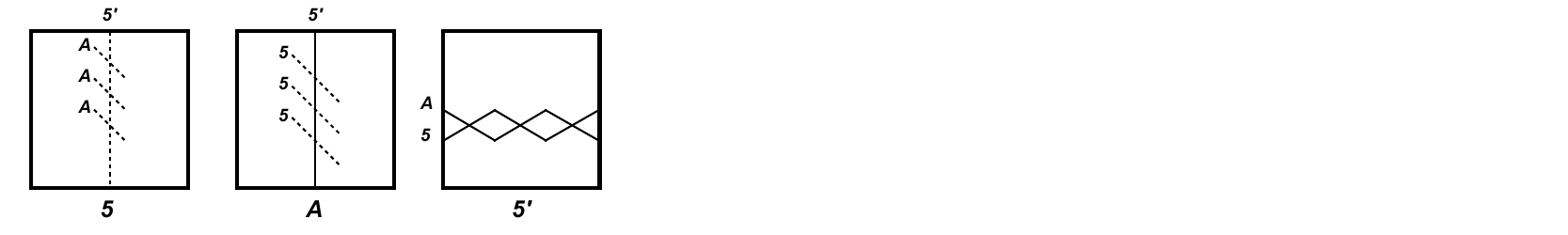}

The last blow-up of the line $L_{A5}$ creates three pinch points on the curve where surfaces  $P_5$ and $P_5'$ intersect.

Concluding, after resolution of generic fibers central fiber has one double curve with four pinch points on it. This double curve lies over the fivefold point.


\subsection{A $p_{4}^{0}$ point degenerates to a $p_{5}^{2}$ point}

We start with the equation

$$xyz(x+y+z)(x+y+w)=0.$$

The central fiber looks as follows.

\includegraphics[scale=0.64]{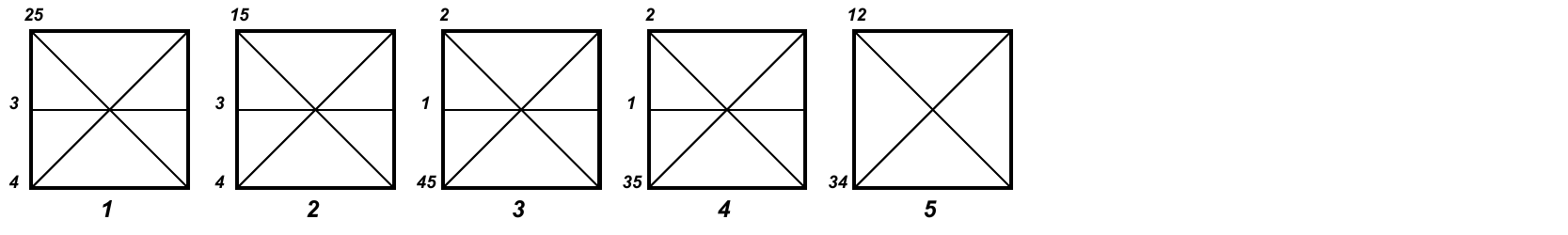}

We first blow up the quadruple point.

\includegraphics[scale=0.64]{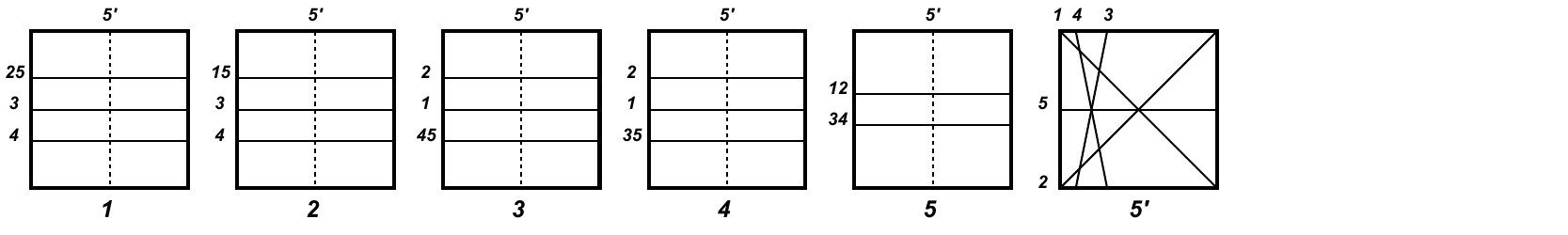}

The only remaining blow-ups are those of double curves, we start with the curves $L_{13},L_{14},L_{23},L_{24}$.

\includegraphics[scale=0.64]{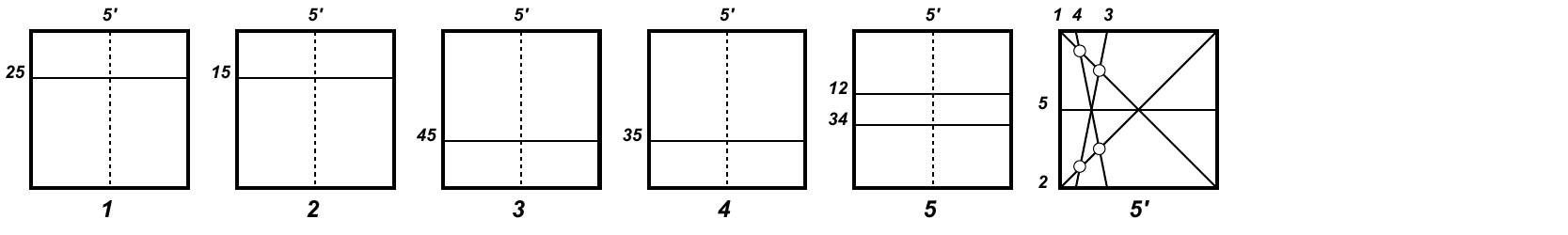}

Blow up the curve $L_{12}$

\includegraphics[scale=0.64]{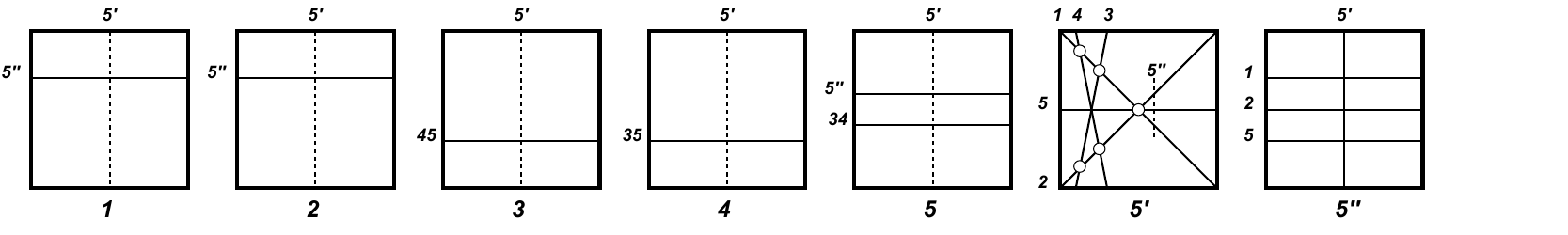}

As the next step, we blow up the curves $L_{15},L_{25}$. This creates two pinch points on the curve where surfaces $P_5'$ and $P_5''$ intersect. The surfaces $P_1,P_2$ are disjoint from the others after that, so we remove them from diagrams.

\includegraphics[scale=0.64]{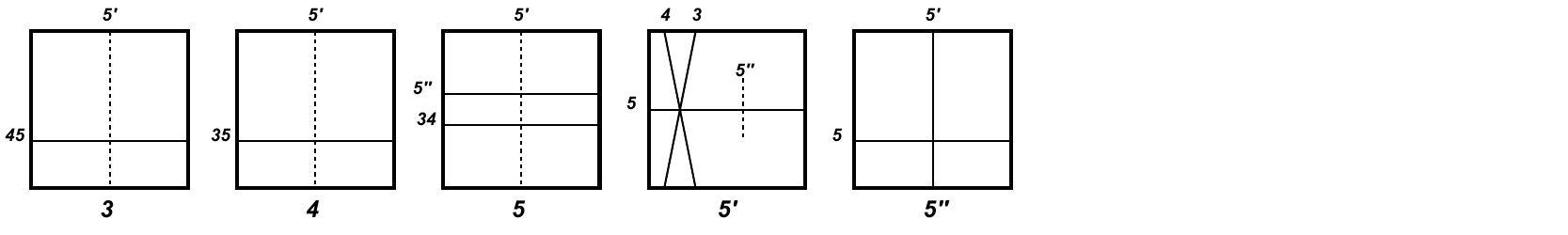}

Blow up the curve $L_{34}$.

\includegraphics[scale=0.64]{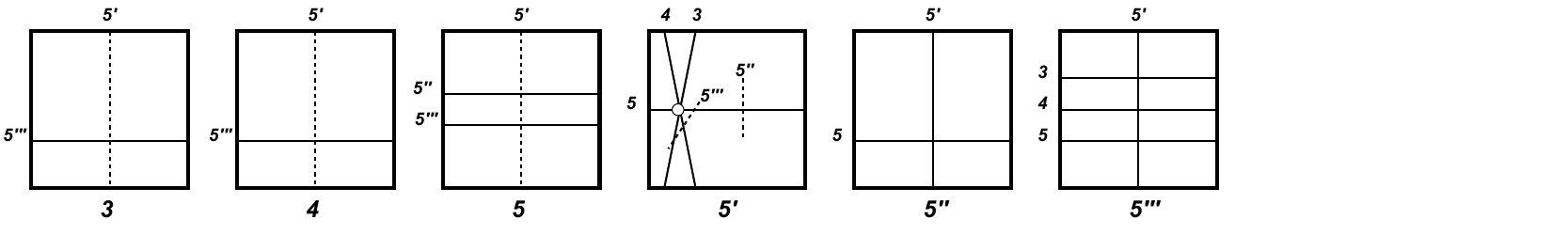}

The last two blow-ups are of the curves $L_{35}$ and $L_{45}$. This will create two pinch points on the curve where surfaces $P_5'$ and $P_5'''$ intersect.

Concluding, after resolving generic fibers there are five double curves in the following configuration. Three curves lying on the surface $P_5'$ are over fivefold point, the two other curves are over two respective triple lines.  The configuration of these 5 double lines can be visualized in the following picture, with dots representing pinch points.

\begin{center}
\begin{tikzpicture}[scale=1.3]
\node at (-0.2,0) {$55'$};
\draw[thick] (0,0) -- (4,0);

\node at (1,1.2) {$5'5''$};
\draw[thick] (1,-1) -- (1,1);

\node at (0,1.2) {$55''$};
\draw[thick] (2,-1) -- (0,1);

\node at (3,1.2) {$5'5'''$};
\draw[thick] (3,-1) -- (3,1);

\node at (2,1.2) {$55'''$};
\draw[thick] (4,-1) -- (2,1);

\filldraw[black] (1,-0.5) circle (2pt);
\filldraw[black] (1,0.5) circle (2pt);

\filldraw[black] (3,-0.5) circle (2pt);
\filldraw[black] (3,0.5) circle (2pt);

\end{tikzpicture}
\end{center}


\subsection{New $p_{4}^{1}$ point}

We start with the equation

$$xy(x+y+w)z=0$$

The central fiber can be presented with this diagram.

\includegraphics[scale=0.64]{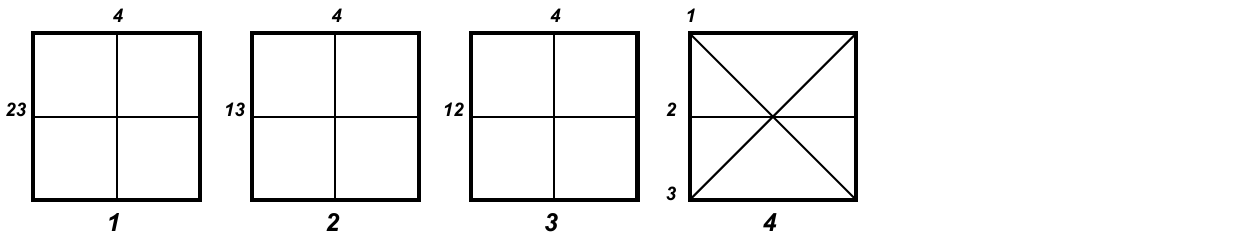}

Branching locus of generic fibers consists of four planes in general position. Therefore, the resolution consists of only blowing up the double lines. We start with blowing up the double line $L_{12}$.

\includegraphics[scale=0.64]{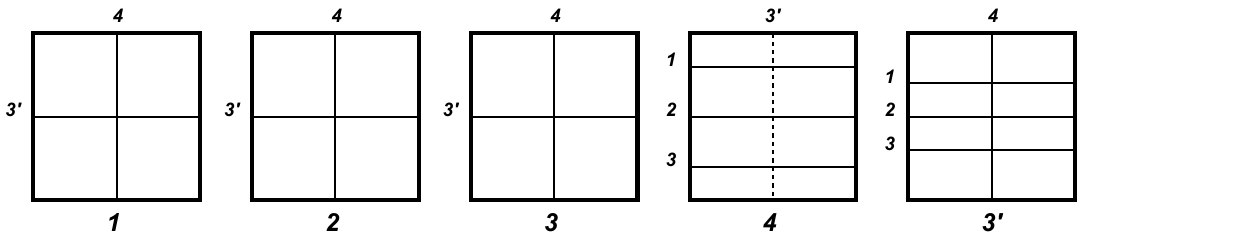}

Next, we blow up the lines $L_{23}, L_{13}$.

\includegraphics[scale=0.64]{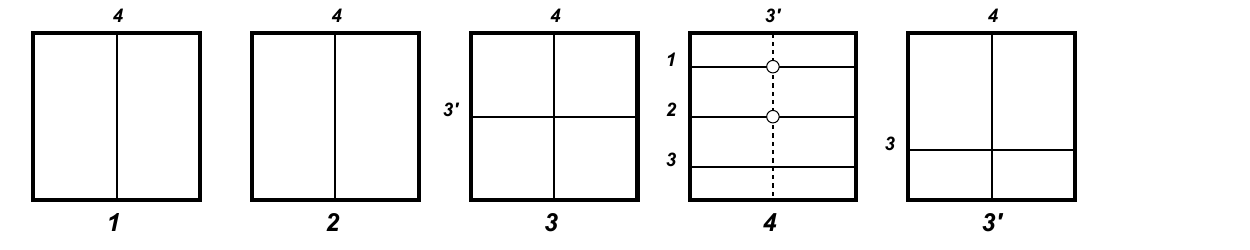}

After that, we blow up the lines $L_{14}, L_{24}$.

\includegraphics[scale=0.64]{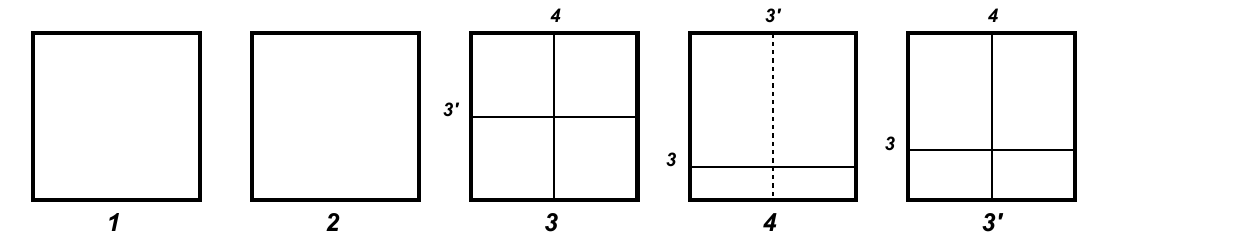}

The only line left to blow up is the line $L_{34}$. It creates a pinch point on the curve where surfaces $P_3$ and $P_3'$ intersect.

After a resolution of a generic fiber in the central fiber we have a single double curve over a triple line with a pinch point on it.

\subsection{A $p_{4}^{0}$ point degenerates to a $p_{4}^{1}$ point}

We start with the following equation. 
$$xy(x+y+zw)z=0$$

The central fiber can be presented with the following diagram.

\includegraphics[scale=0.64]{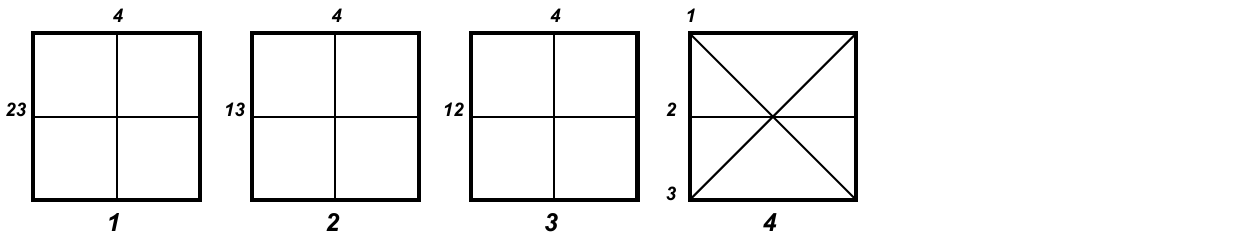}

The resolution starts with blowing up the fourfold point.

\includegraphics[scale=0.64]{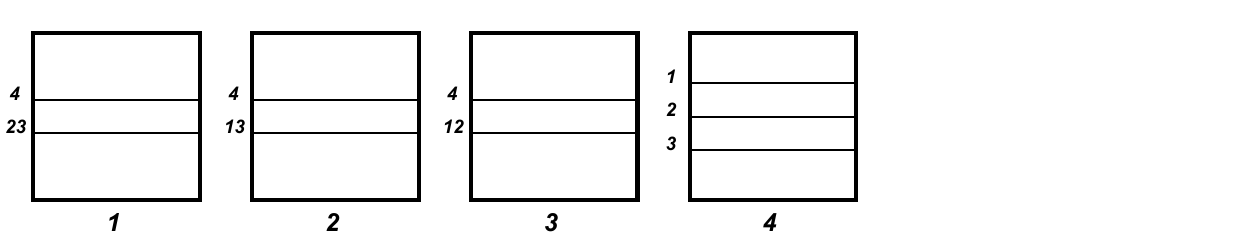}

Now, the only blow-ups left are those of all double lines. Apart from a new triple line $L_{123}$ there are no new singularities in the central fiber. Therefore, after resolving generic fiber there is a double curve in the central fiber over a triple line.

\subsection{A $p_{4}^{0}$ point degenerates to a $p_{5}^{1}$ point}

We start with the following equation.

$$xyz(x+y+z)(x-y+w)=0$$

The central fiber can be presented with the following diagram.

\includegraphics[scale=0.64]{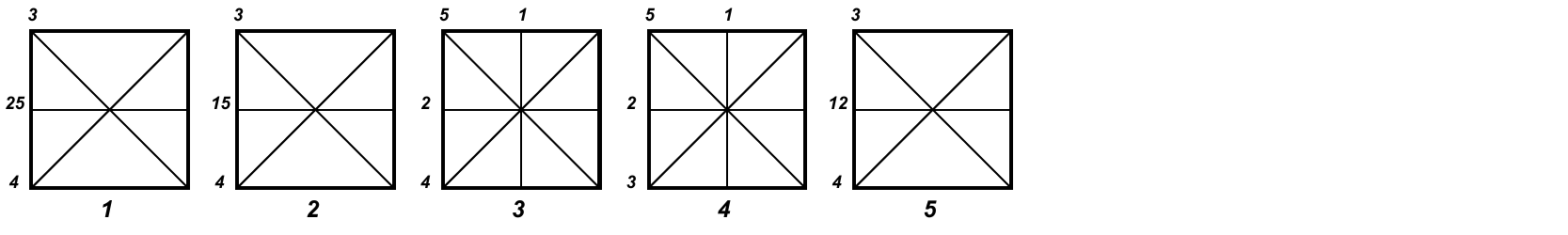}

We start with blowing up the fourfold point.

\includegraphics[scale=0.64]{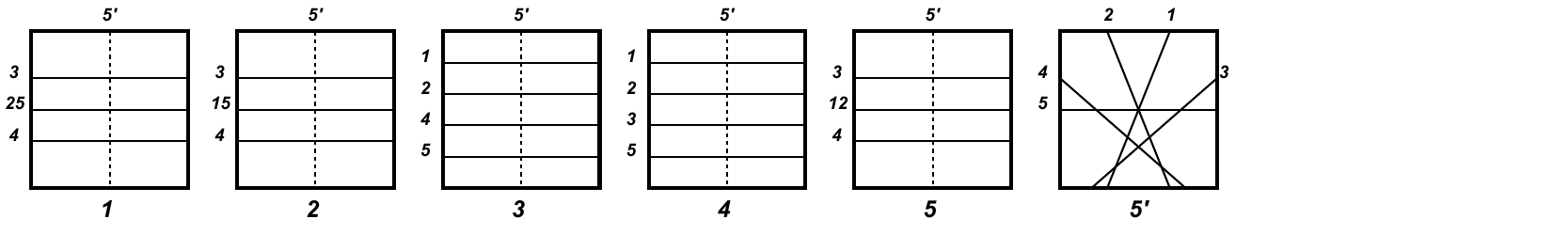}

Now, we blow up all the double curves. We first blow up $L_{35}$ and $L_{45}$ which creates two pinch points on the curve where surfaces $P_5$ and $P_5'$ intersect.

\includegraphics[scale=0.64]{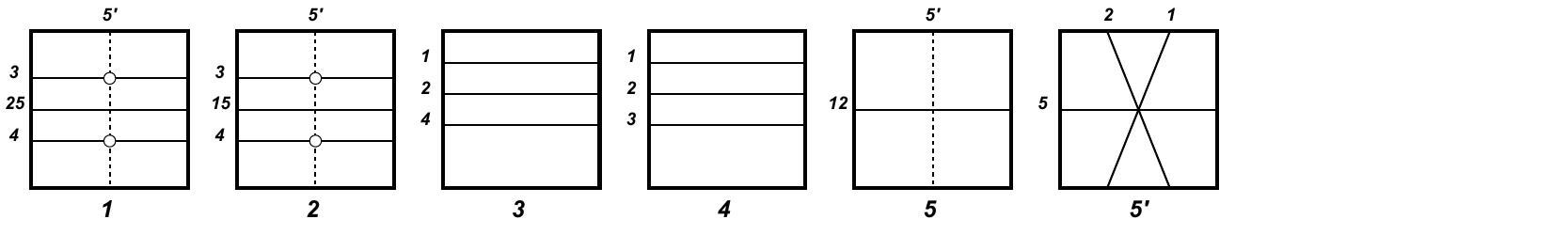}

Blow up curves $L_{13},L_{14},L_{23},L_{24},L_{34}$.

\includegraphics[scale=0.64]{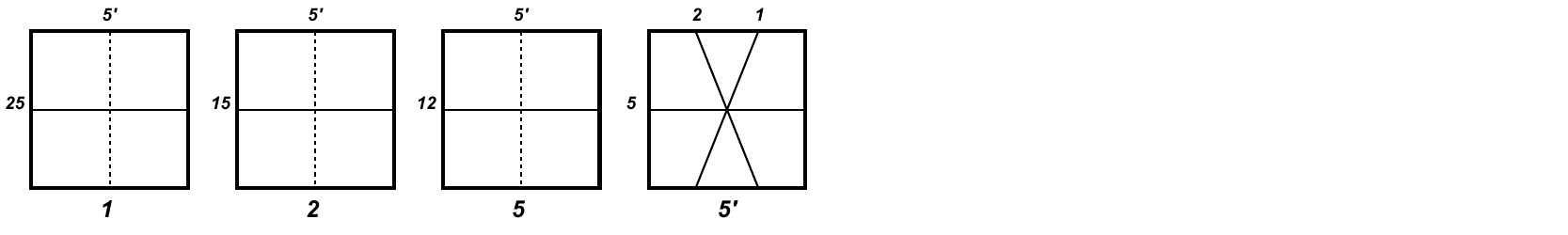}

Blow up the curve $L_{12}$.

\includegraphics[scale=0.64]{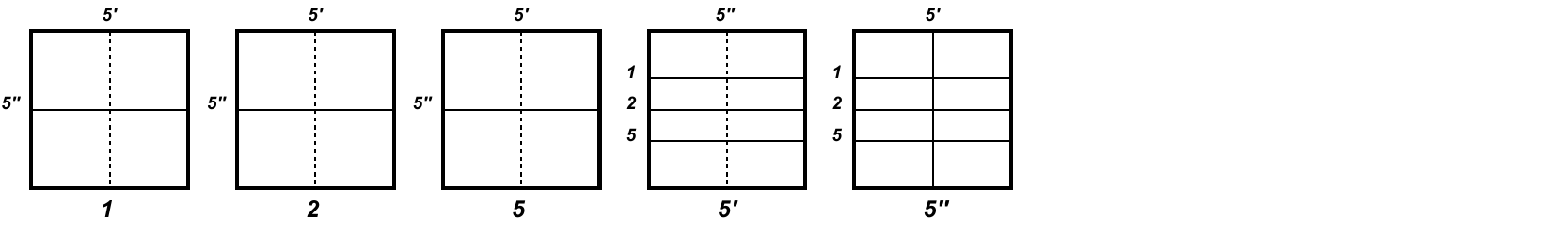}

The last blow-ups are those of curves $L_{15}$ and $L_{25}$. These blow-ups create two pinch points on the curve where surfaces $P_5'$ and $P_5''$ intersect.

Concluding, after resolving generic fibers there are three double curves intersecting in a single point in the degenerate fiber. Two of these curves are over the fivefold point and each of them has two pinch points. The third curve is over the triple line.

\subsection{A $p_{5}^{0}$ point degenerates to a $p_{5}^{2}$ point}

We start with the equation.
$$xyz(x+y+wz)(x+wy+z)=0$$

The central fiber can be presented with the following diagram.

\includegraphics[scale=0.64]{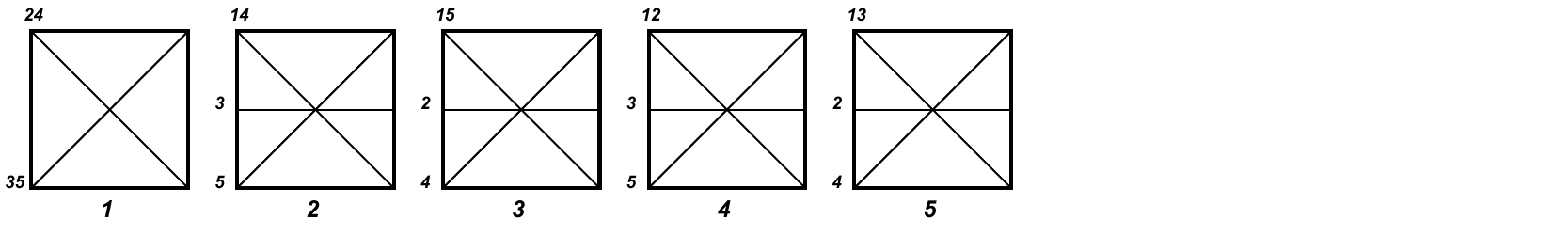}

The first blow-up is the one of the fivefold point.

\includegraphics[scale=0.64]{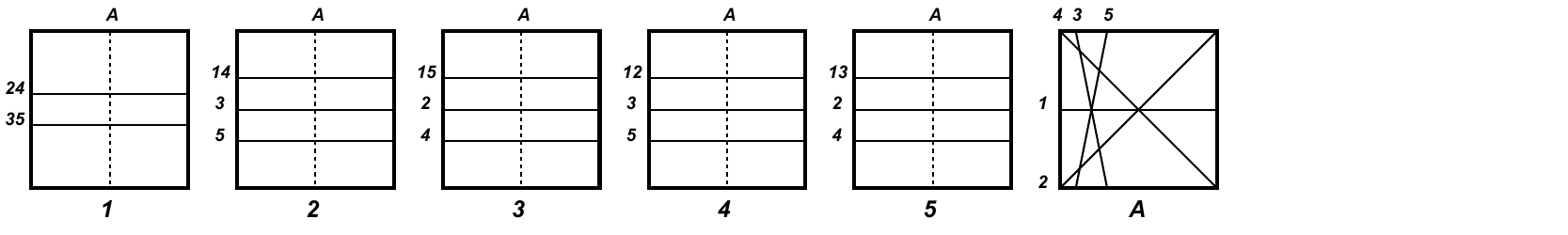}

Now we proceed with blowing up all double lines, we start with $L_{1A}$.

\includegraphics[scale=0.64]{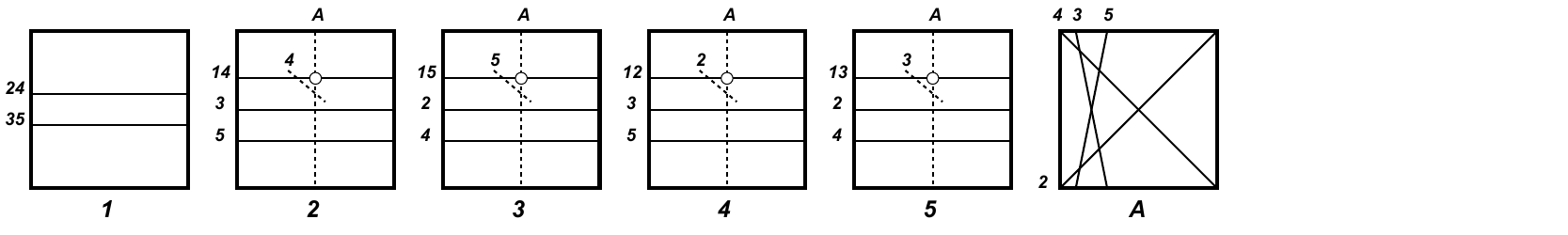}

Then, we blow up the lines $L_{2A},L_{3A},L_{4A},L_{5A}$ (in this order). This leaves the surface $P_A$ disjoint from the other ones.

\includegraphics[scale=0.64]{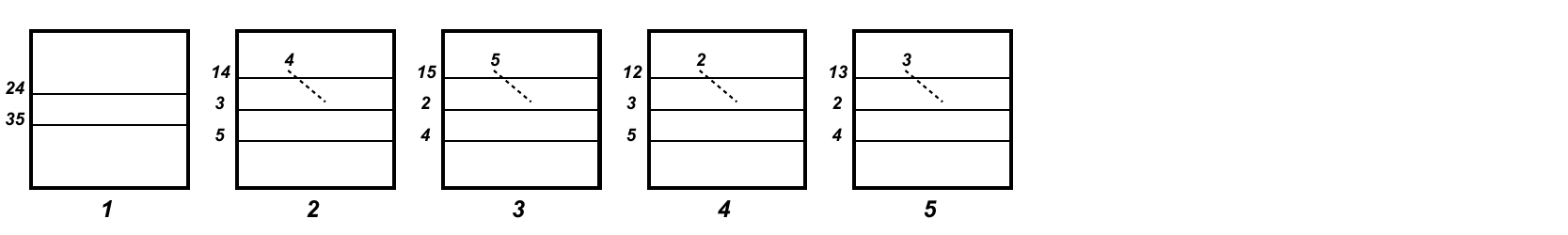}

Now, we blow up the lines $L_{23},L_{25},L_{34},L_{45}$.

\includegraphics[scale=0.64]{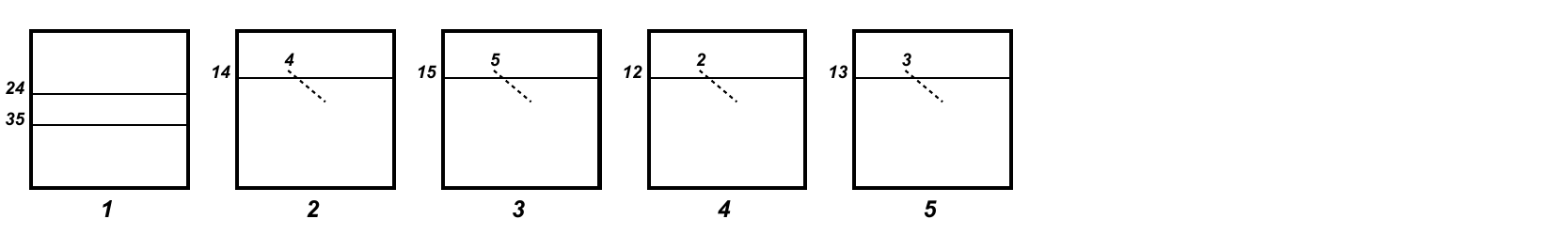}

We proceed by blowing up the lines $L_{14}, L_{15}$.

\includegraphics[scale=0.64]{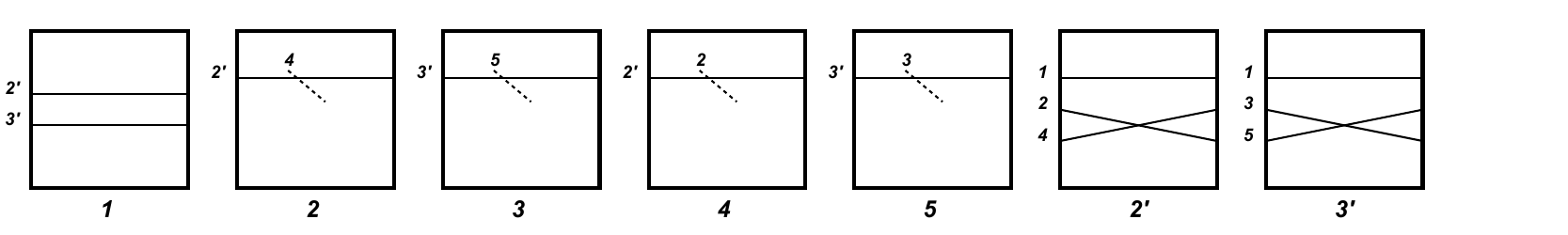}

The final blow-up of the lines $L_{12}, L_{13}, L_{24}, L_{35}$ creates one pinch point on the curve where surfaces $P_2$ and $P_2'$ intersect, and the other one on the curve where surfaces $P_3$ and $P_3'$ intersect.

Concluding, after resolving generic fibers, central fiber has two disjoint double lines with one pinch point on each of them.

\subsection{A $p_{5}^{0}$ point degenerates to a $p_{5}^{1}$ point}

We start with the equation.
$$xyz(x+y+wz)(x+2y+z)=0$$

The central fiber can be presented with the following diagram.

\includegraphics[scale=0.64]{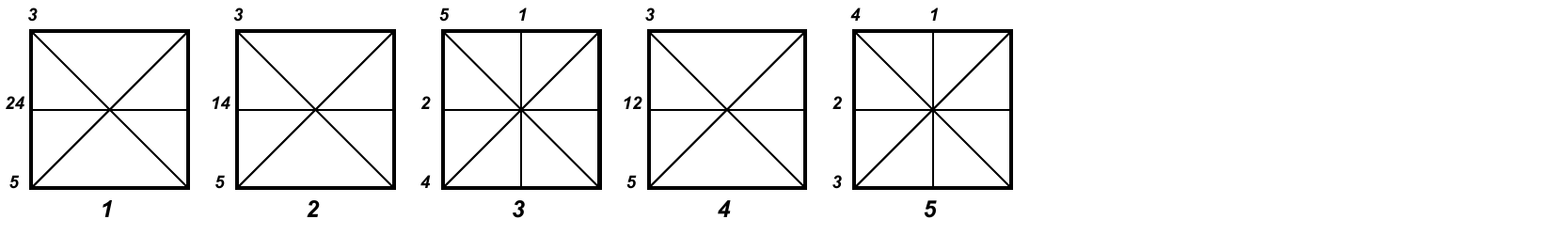}

The first blow-up of the fivefold point has the following effect.

\includegraphics[scale=0.64]{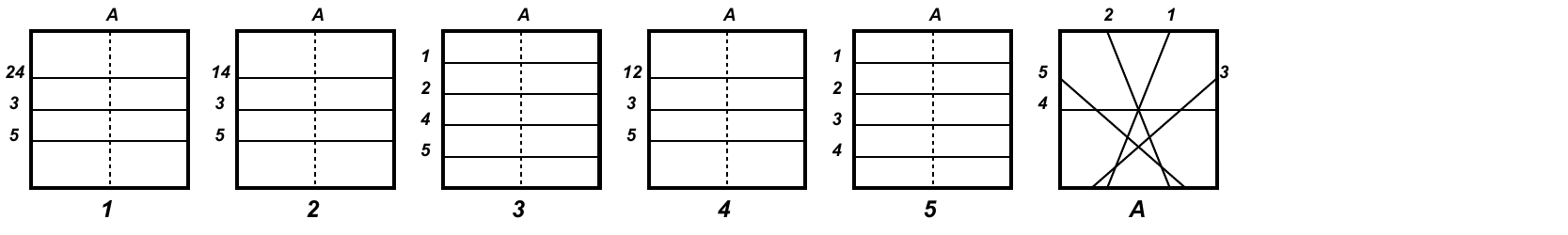}

Next we blow up the curve $L_{1A}$.

\includegraphics[scale=0.64]{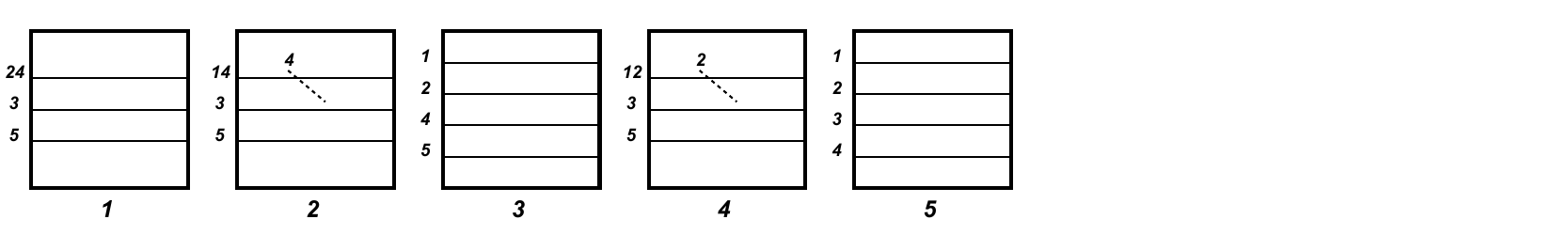}

We blow up the lines $L_{1A},L_{3A},L_{4A},L_{5A}$, the surface $P_A$ becomes disjoint from the other ones.

\includegraphics[scale=0.64]{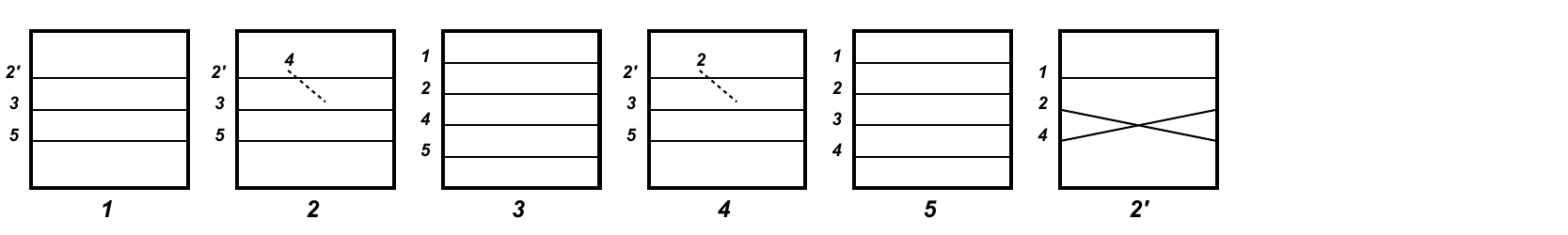}

By blowing up all remaining double lines from generic fibers we create one pinch point on a curve where surfaces $P_2$ and $P_2'$ intersect.

Concluding, after resolving generic fibers there is one double curve in the degenerate fiber. This double curve has a single pinch point and lies over the triple line.

\section{Semistable reduction}
Let $\mc{X} \rr \Delta$ be a family of singular double covers of $\PP^3$ branched over an octic arrangement, let $\mc{Y} \rr \Delta$ be a family coming from resolving generic fibers of $\mc{X}$, we will refer to $\mc{Y}$ as a family of \emph{generically resolved double octics}. Generic fibers of $\mc{Y}$ are smooth double octics while the degenerate fiber is a threefold which might have singularities described in the previous section. In this section we give a construction of a semistable degeneration of $\mc{Y}$.

\begin{definition}\cite[p.102]{Morrison}
A semistable family $\pi \colon \mc{S} \rr \Delta$ is a proper family, where $\mc{S}$ is a smooth variety, whose generic fibers over $\Delta^*$ are smooth and the central fiber over $0$ is a union of smooth varieties intersecting transversally along smooth subvarieties.
\end{definition}

A precise meaning of the semistable reduction is given by the following theorem.

\begin{theorem}\cite[p.53]{semistable}
Given a degeneration $\pi \colon \mc{Y} \rr \Delta$ there exists a base change $b \colon \Delta \rr \Delta, t \rr t^n$, a semistable family $\psi \colon \mc{S} \rr \Delta$ and a commutative diagram

\begin{center}
\begin{tikzcd}
\mc{S} \arrow[r,dashed]{}{f} \arrow[rd]{}{\psi} & \mc{Y} \times_b \Delta \arrow[r] \arrow[d] & \mc{Y} \arrow[d] \\
                    & \Delta \arrow[r]{}{b} & \Delta          
\end{tikzcd}
\end{center}

where $f \colon \mc{S} \rr \mc{Y} \times_b \Delta$ is a composition of blow-ups and blow-downs along subvarieties in the central fiber.
\end{theorem}

In the proof of the following lemma we describe blow-ups and a base change that give a semistable reduction $\mc{S}$ of $\mc{Y}$.

\begin{proposition}
    Let $\mc{Y} \rr \Delta$ be a family of generically resolved double octics. There exists a semistable family $\mc{S} \rr \Delta$ whose smooth fibers are isomorphic with those of $\mc{Y}$ and whose central fiber is a union of a smooth Calabi-Yau threefold and other components which are either quadric bundles over $\PP^1$ or $\PP^1 \times \PP^1 \times \PP^1$ blown-up at a finite number of points.
\end{proposition}

\begin{proof}
Degenerate fibers of $\mc{Y}$ can have only the following local types of singularities. Here, the term \emph{local} means that we take equations up to an analytic biholomorphism of a neighbourhood of a singularity.

\begin{enumerate}
    \item Node: $xy=zw$.
    \item Double line: $u^2=xy$.
    \item Three double lines intersecting at a point: $u^2 = xyz$.
    \item Pinch point on a double line: $u^2=zw^2$.
\end{enumerate}

In section \ref{diagrams} we have shown that nodes in degenerate fibers constructed in previous sections lie on a smooth surface isomorphic to $\PP^1 \times \PP^1$, therefore, they have a small resolution. The new components after a blow-up of such surface is $\PP^1 \times \PP^1 \times \PP^1$ blown-up in $2d$ points, where $d$ is the number of nodes that lied on that surface.

The other type of singularities consists of double curves. In each of these cases we blow up these curves in some chosen order. The central fiber after that is of the form $\wdt{Y} \cup Q_1 \cup ... \cup Q_n$, where $\wdt{Y}$ is a double octic and $Q_i$ are non-reduced $\PP^2$ bundles with multiplicity 2. The intersections $Q_i \cap \wdt{Y}$ are smooth conic bundles which have singular fibers over pinch points. Let us call the blow-up of $\mc{Y}$ along these singularities in the central fiber as $\mc{Y}'$.

The family $\mc{S}$ is defined as a double cover of $\mc{Y}'$ branched along reduced components of the central fiber of $\mc{Y}'$. This can be done by taking a base change $b \colon t \rr t^2$. The reduced components of the central fiber of $\mc{S}$ are the same as those of $\mc{Y'}$, while $\PP^2$ are replaced by quadric bundles. Therefore, the components of $\mc{S}$ are as in the lemma.

\end{proof}

We finish this section by proving that the total space of $\mc{S}$ is smooth as a fourfold over $\CC$. Our strategy is to show that blow-ups we perform while resolving generic fibers also resolve all singularities of the total space.

Let $\mc{X}$ be a double cover branched along a union of smooth hypersurfaces $\mc{D}= \bigcup\limits^n_{i=1} \mc{D}_i$ such that intersections of any subset of $\mc{D}$ are smooth.

Let us define $\mc{A}$ to be the set of irreducible components of intersections $\cap_{j \in J} D_j$ for any subset $J \subset \{1, ..., n\}$. Then, for $C \in \mc{A}$, we define $m(C) := \#\{i : C \subset D_i\}$.
A subvariety $C \in \mc{A}$, with $m(C) \geq 3$, is called near-pencil if it is contained in another intersection $C' \in \mc{A}$ with $dim(C) = dim(C') - 1$ and $m(C) = m(C')+1$.

In the section \ref{extending resolution} we proved that one-parameter families of double covers branched along octic arrangements have a structure of a double cover of $\PP^4$ branched along a union of eight hypersurfaces. This allows us to use the following theorem.

\begin{theorem}[\cite{li}]
\label{ingalis}
Assume that for every subvariety $C$ of double cover $X$, where $m(C) \geq 2$, $C$ is a smooth transversal intersection of components in the branch locus of $X$. If every such $C$ is near-pencil or satisfies $\floor{\frac{m(C)}{2}} = dim(X) - dim(C) -1$, then $X$ admits a crepant resolution.
\end{theorem}

Using the above theorem we prove the following.

\begin{proposition}
    Let $\mc{X}$ be a family of double covers branched along a family of octic arrangements, such that the degenerate fiber is also an octic arrangement. Let $\sigma \colon \mc{Y} \rr \mc{X}$ be sequence of blow-ups in the extended resolution from the section \ref{extending resolution}. Then, the fourfold $\mc{Y}$ is smooth.
\end{proposition}

\begin{proof}
We first need to show that blow-ups in the extended resolution have smooth centers. In an analytic neighbourhood of a generic fiber, centers of blow-ups are of the form $V \times \Delta$, where $V$ is a singularity of a generic fiber. Since $V$ satisfies conditions from theorem \ref{ingalis}, $V \times \Delta$ does as well. In neighbourhoods of degenerate fibers we use classification of combinatorial types of degenerations from the appendix \ref{singularity classification}. In all but four cases considered in section \ref{diagrams} the local branch divisor of the fourfold consists of a union of hyperplanes in $\CC^4$, therefore its intersections are smooth and transversal. In the other four cases some hypersurfaces in the branch divisor have degree two, and some of intersection between them might be singular or non-transversal. These degenerations can be given by the following sample equations taken from the previous section:

\begin{itemize}
    \item $x y(x+y+zw) z=u^{2}$  - the degeneration 4.8, 
    \item $x y(x+y)z(x+yw+z)=u^{2}$ - the degeneration 4.3,
    \item $x y z(x+y+zw)(x+yw+z)=u^{2}$ - the degeneration 4.10,
    \item $x y z(x+y+zw)(x+2 y+z)=u^{2}$ - the degeneration 4.11.
\end{itemize}

In each of these cases the first blow-up in the resolution is a blow-up of a curve $x=y=z=0$, which comes from blowing-up quadruple or quintuple points in generic fibers. We claim that this blow-up makes future centers of blow-ups smooth.

We start with the local degeneration $x y(x+y+z w) z=u^{2}$. Here, the only singular intersection is $x=y=x+y+zw=0$. The first blow-up is the blow-up of a quadruple point in generic fibers, this means that the whole family is blown up in the quadruple line $x=y=z=0$. The equations of this blow-up are: $x y(x+y+z w) z=x p_{1}-y p_{0}=x p_{2}-z p_{0}=y p_{2}-z p_{1}=0$, where $\left(p_{0}: p_{1}: p_{2}\right) \in \mathbb{P}^{2}$. One can easily check that strict transforms of surfaces $x=0, y=0, x+y+z w=0$ intersect in a smooth curve $x=y=w=p_{0}=p_{1}=0$. The other intersections remain smooth, as the blow-up had a smooth centre.

In case of the other three local degenerations, the same argument shows that the first blow-up make future centers of blow-ups smooth.

To finish the proof we show that singularities of the fourfold which are contained in degenerate fibers are near-pencil. Such singularities come from new triple lines, new fourfold points and new quintuple points. In case of a triple curve in a fourfold, it is always contained in a double surface. For a quadruple point in a fourfold, it is always contained in a triple curve. For quintuple points, we have to observe that in all local degenerations new quintuple points degenerate from quadruple points in generic fibers, and thus, lie on quadruple curves.

Concluding, after the first step of a resolution all families describing local degenerations satisfy the conditions of the lemma, furthermore a resolution of the family from the section \ref{extending resolution} agrees with the one that resolves the family as a fourfold.
\end{proof}

\section{Monodromy weight spectral sequence}

In this section we shall apply the results of previous sections to compute the monodoromy weight spectral sequence for a semistable degeneration $\mc{S}$ of one-parameter family of double octic Calabi-Yau threefolds. Let us call the components of the central fiber $S_0$ as $S_0 = Z_1 + ... + Z_n$, where $Z_i$ are smooth reduced prime divisors in $S_0$. We define for $p=1,2,...$;

\[
S^{[p]}=\bigsqcup_{1\leq i_1 \leq ... \leq i_p \leq n} Z_{i_1,...,i_p} = \bigsqcup_{1\leq i_1 \leq ... \leq i_p \leq n} Z_{i_1} \cap ... \cap Z_{i_p},
\]

the disjoint union of $p$-fold intersections of components of $S_0$. 

The monodromy weight spectral sequence for the semistable family $\pi \colon \mc{S} \rr \Delta$ has the $E_1$ page;

\[
E^{-k,h+k}_1=\oplus_{j\geq \max\{-k,0\}}H^{h-2j-k}(S^{[2j+k+1]})[-j -k],
\]
(for details see \cite[p.273]{steenbrink}).

This spectral sequence converges to the limiting mixed Hodge structure $H^q(S_\infty)$.

\begin{theorem}[{\cite[Cor. 11.23]{steenbrink}}]
    The monodromy weight spectral sequence degenerates at $E_2$ and
    \[
    E_1^{-p,q+p} \implies H^q(S_\infty,\QQ).
    \]
\end{theorem}

Moreover, on the level of vector spaces the limiting Hodge structure agrees with Hodge structure of a generic fiber.

\begin{proposition}[{\cite[Cor. 11.25]{steenbrink}}]
    $\operatorname{dim}F^p H^k(X_\infty) = \operatorname{dim}F^p H^k(X_t),t 
    \in \Delta^*$
\end{proposition}

Using results from the previous sections we show how one computes this sequence for a semistable family of double octics.

\subsection{Degeneration at $w=1$ of the arrangement No. 2.}
\label{ex1}
In this degeneration the only new incidence is given by the new fourfold point. As we have seen in section \ref{diagrams} after a resolution of a generic fiber, there are two nodes in the degenerate fiber. The surface which is a pullback of any of the last two blown-up curves contains both nodes. We blow up this surface to resolve these nodes. After this blow-up we have the following varieties appearing in the spectral sequence.

\begin{itemize}
  \item $Y$ - smooth Calabi-Yau threefold.

  \item $Q_{1}$ - smooth $\mathbb{P}^{1}$ bundle.

  \item $Y \cap Q_{1}$ - blow-up of $\mathbb{P}^{1} \times \mathbb{P}^{1}$ at two points.

\end{itemize}

The spectral sequence has the following entries.

\[\begin{tikzcd}[row sep=0 em]
	H^4(Y \cap Q_1) & H^{6}\left(Y\right) \oplus H^6(Q_1) & 0 \\
	H^3(Y \cap Q_1) & H^{5}\left(Y\right) \oplus H^5(Q_1)  & 0\\
        H^2(Y \cap Q_1) & H^{4}\left(Y\right) \oplus H^4(Q_1)  & H^4(Y \cap Q_1)\\
        H^1(Y \cap Q_1) & H^{3}\left(Y\right) \oplus H^3(Q_1)  & H^3(Y \cap Q_1)\\
        H^0(Y \cap Q_1) & H^{2}\left(Y\right) \oplus H^2(Q_1)  & H^2(Y \cap Q_1)\\
        0 & H^{1}\left(Y\right) \oplus H^1(Q_1)  & H^1(Y \cap Q_1)\\
        0 & H^{0}\left(Y\right) \oplus H^0(Q_1)  & H^0(Y \cap Q_1)\\
	\arrow[from=1-1, to=1-2]
	\arrow[from=1-2, to=1-3]
	\arrow[from=2-1, to=2-2]
	\arrow[from=2-2, to=2-3]
	\arrow[from=3-1, to=3-2]
	\arrow[from=3-2, to=3-3]
	\arrow[from=4-1, to=4-2]
	\arrow[from=4-2, to=4-3]
	\arrow[from=5-1, to=5-2]
	\arrow[from=5-2, to=5-3]
	\arrow[from=6-1, to=6-2]
	\arrow[from=6-2, to=6-3]
	\arrow[from=7-1, to=7-2]
	\arrow[from=7-2, to=7-3]
\end{tikzcd}\]

On the level of complex vector spaces this is equal to.

\[\begin{tikzcd}[row sep=0 em]
	\CC & \CC \oplus \CC & 0 \\
	0 & 0  & 0\\
        \CC^4 & \CC^{70} \oplus \CC^3  & \CC \\
        0 & \CC^2  & 0 \\
        \CC & \CC^{70} \oplus \CC^3  & \CC^4 \\
	0 & 0  & 0\\
        0 & \CC \oplus \CC & \CC \\
	\arrow[from=1-1, to=1-2]
	\arrow[from=1-2, to=1-3]
	\arrow[from=2-1, to=2-2]
	\arrow[from=2-2, to=2-3]
	\arrow[from=3-1, to=3-2]
	\arrow[from=3-2, to=3-3]
	\arrow[from=4-1, to=4-2]
	\arrow[from=4-2, to=4-3]
	\arrow[from=5-1, to=5-2]
	\arrow[from=5-2, to=5-3]
	\arrow[from=6-1, to=6-2]
	\arrow[from=6-2, to=6-3]
	\arrow[from=7-1, to=7-2]
	\arrow[from=7-2, to=7-3]
\end{tikzcd}\]

The generators of $H^2(Y \cap Q_1)$ consist of two rulings and two exceptional curves. The images of the classes of these two exceptional curves are equal in 
$H^4(Q_1)$ and $H^4(Y)$ and are non-zero. The map in the third row $H^{4}\left(Y\right) \oplus H^4(Q_1) \rr H^4(Y \cap Q_1)$ is surjective. Therefore, vector spaces on the second page are the following.

\[\begin{tikzcd}[row sep=0 em]
	0 & \CC & 0 \\
	0 & 0  & 0\\
        \CC & \CC^{69} & 0 \\
        0 & \CC^2  & 0 \\
        0 & \CC^{69} & \CC \\
	0 & 0  & 0\\
        0 & \CC & 0 \\
	\arrow[from=1-1, to=1-2]
	\arrow[from=1-2, to=1-3]
	\arrow[from=2-1, to=2-2]
	\arrow[from=2-2, to=2-3]
	\arrow[from=3-1, to=3-2]
	\arrow[from=3-2, to=3-3]
	\arrow[from=4-1, to=4-2]
	\arrow[from=4-2, to=4-3]
	\arrow[from=5-1, to=5-2]
	\arrow[from=5-2, to=5-3]
	\arrow[from=6-1, to=6-2]
	\arrow[from=6-2, to=6-3]
	\arrow[from=7-1, to=7-2]
	\arrow[from=7-2, to=7-3]
\end{tikzcd}
\]

As a consequnece we recover the Betti numbers of a generic fiber $S_{\eta}$ of the family: $b^{0}=1$, $b^{1}= 0$, $b^{2}=69, b^{3}=4, b^{4}=69, b^{5}=0, b^{6}=1$ as well as a mixed Hodge structure on $H^3(S_\infty)$ which is non pure.

\subsection{Degeneration at $w=1$ of the arrangement No. 34.}
\label{ex2}
In this degeneration the only new incidence is a new $p_{5}^{1}$ point coming from a collision of two $p_{4}^{1}$ points. After a resolution of a generic fiber the degenerate fiber has a double line with four pinchpoints. After passing to semistable degeneration we have the following varieties appearing in the spectral sequence.

\begin{itemize}
  \item $Y$ - smooth Calabi-Yau threefold.

  \item $Q_{1}$ - double cover of $\PP^2 \times \PP^1$ branched along $Y \cap Q_{1}$.

  \item $Y \cap Q_{1}$ - conic bundle with a generic fiber being smooth while fibers over each of the four pinch points consist of two intersecting lines.

\end{itemize}

On the level of complex vector spaces we get the $E_1$ of the following spectral sequence.

\[\begin{tikzcd}[row sep=0 em]
	\CC & \CC \oplus \CC & 0 \\
	0 & 0  & 0\\
        \CC^6 & \CC^{54} \oplus \CC^2 & \CC \\
        0 & \CC^2 \oplus \CC^2 & 0 \\
        \CC & \CC^{54} \oplus \CC^2 & \CC^6 \\
	0 & 0  & 0\\
        0 & \CC \oplus \CC & \CC \\
	\arrow[from=1-1, to=1-2]
	\arrow[from=1-2, to=1-3]
	\arrow[from=2-1, to=2-2]
	\arrow[from=2-2, to=2-3]
	\arrow[from=3-1, to=3-2]
	\arrow[from=3-2, to=3-3]
	\arrow[from=4-1, to=4-2]
	\arrow[from=4-2, to=4-3]
	\arrow[from=5-1, to=5-2]
	\arrow[from=5-2, to=5-3]
	\arrow[from=6-1, to=6-2]
	\arrow[from=6-2, to=6-3]
	\arrow[from=7-1, to=7-2]
	\arrow[from=7-2, to=7-3]
\end{tikzcd}\]

The left map in the first row is a injective, by duality the right map in the last row is surjective. The map in the fifth row $\mathbb{C} \rightarrow \mathbb{C}^{54} \oplus \mathbb{C}^{2}$ is injective as it comes from a closed embedding of an algebraic cycle, the dual map $\mathbb{C}^{54} \oplus \mathbb{C}^{2} \rightarrow \mathbb{C}$ is thus surjective.

The second page of the spectral sequence is the following.

\[\begin{tikzcd}[row sep=0 em]
	0 & \CC & 0 \\
	0 & 0  & 0\\
        0 & \CC^{49} & 0 \\
        0 & \CC^4  & 0 \\
        0 & \CC^{49} & 0 \\
	0 & 0  & 0\\
        0 & \CC & 0 \\
	\arrow[from=1-1, to=1-2]
	\arrow[from=1-2, to=1-3]
	\arrow[from=2-1, to=2-2]
	\arrow[from=2-2, to=2-3]
	\arrow[from=3-1, to=3-2]
	\arrow[from=3-2, to=3-3]
	\arrow[from=4-1, to=4-2]
	\arrow[from=4-2, to=4-3]
	\arrow[from=5-1, to=5-2]
	\arrow[from=5-2, to=5-3]
	\arrow[from=6-1, to=6-2]
	\arrow[from=6-2, to=6-3]
	\arrow[from=7-1, to=7-2]
	\arrow[from=7-2, to=7-3]
\end{tikzcd}\]

Once again we recover betti numbers of generic fiber of the family, as well as a mixed Hodge structure on $H^3(S_\infty)$ which is this time pure.

\subsection{Degeneration at $w=1$ of the arrangement No. 273.}
\label{ex3}
In this degeneration after resolving generic fibers we have the following singularities in the central fiber:
    
\begin{itemize}
  \item two double curves with no pinch points,

  \item five double curves with two pinch points,

  \item three points where three double curves meet.

\end{itemize}

Their configuration can be visualized in the following picture.
\begin{center}
\begin{tikzpicture}[scale=1.3]

\node at (-2.2,4) {$l_1$};
\draw[thick] (-2,4) -- (5,4);

\node at (-1,5.2) {$l_2$};
\draw[thick] (-1,5) -- (-1,-1);

\node at (-2.2,5.2) {$l_3$};
\draw[thick] (-2,5) -- (0,3);
\filldraw[black] (-1.5,4.5) circle (2pt);
\filldraw[black] (-0.5,3.5) circle (2pt);

\node at (-2.2,1.2) {$l_6$};
\draw[thick] (-2,1) -- (0,1);
\filldraw[black] (-1.5,1) circle (2pt);
\filldraw[black] (-0.5,1) circle (2pt);

\node at (-2.2,2.2) {$l_7$};
\draw[thick] (-2,2) -- (0,0);
\filldraw[black] (-0.5,0.5) circle (2pt);
\filldraw[black] (-1.5,1.5) circle (2pt);

\node at (3,5.2) {$l_4$};
\draw[thick] (3,5) -- (3,3);
\filldraw[black] (3,4.5) circle (2pt);
\filldraw[black] (3,3.5) circle (2pt);

\node at (2,5.2) {$l_5$};
\draw[thick] (2,5) -- (4,3);
\filldraw[black] (2.5,4.5) circle (2pt);
\filldraw[black] (3.5,3.5) circle (2pt);
\end{tikzpicture}
\end{center}

Straight lines represent double curves, and the dots represent pinch points on them. This time the construction of a semistable degeneration is more complicated.

We begin by performing two blow-ups: first along the curve  $l_1$, followed by the blow-up of curve $l_2$. Next, we blow up the remaining five double curves ($l_3$ through $l_7$), which become disjoint after blow-ups of $l_1$ and $l_2$. We analyze the local situation above each triple point.

Let us take the point where lines $l_1,l_2,l_3$ intersect. By blowing up these three lines we introduce three new components of the central fiber $R_1,R_2,R_3$. These three new components are non-reduced, therefore we take a double cover branched along the strict transform of $Y$ - the variety we started with, denote the double covers of $R_1,R_2,R_3$ by $Q_1,Q_2,Q_3$ respectively. Threefolds $R_2,R_3$ are $\PP^2$ bundles which intersect $Y$ along a smooth conic bundles, thus $Q_2,Q_3$ are smooth quadric bundles.

The geometry of $R_1$ and $Q_1$ is slightly more complicated. Generic fiber of $R_1$ is $\PP_2$ and a degenerate fiber has three intersecting components: two projective planes $A,B$ and one projective plane blown up at two points $C$. The projective planes $A,B$ intersect $C$ along its exceptional divisors. To understand geometry of the degenerate fiber of $Q_1$ we visualize the intersection $R_1 \cap Y$ with the following diagram.

\begin{center}
\begin{tikzpicture}[scale=1.3]
\node at (-2,0) {$A \cap Y$};
\draw[very thick](-2,0) circle (1);
\node at (0,0.2) {$C \cap Y$};
\draw[thick] (-1,0) -- (1,0);
\node at (2,0) {$B \cap Y$};
\draw[very thick](2,0) circle (1);

\end{tikzpicture}
\end{center}

Here, $A \cap Y$ and $B \cap Y$ are smooth conics, while $C \cap Y$ is a double line. The double cover of $A,B$ and $C$, branched along these intersections, consists of four surfaces. Two disjoint quadrics $\wb{A}$ and $\wb{B}$ coming from double covers of $A$ and $B$ branched along smooth conics, and two projective planes $\wb{C_1}$ and $\wb{C_2}$ coming from a double cover of $C$ branched along a double line, each of these two planes are blown-up at two points. Moreover the planes $\wb{C_1}$ and $\wb{C_2}$ intersect quadrics $\wb{A}$ and $\wb{B}$, generating two distinct rulings on both quadrics. Finally the other threefolds $Q_2,Q_3$ intersect $Q_1$ along the quadrics $\wb{A}$ and $\wb{B}$ respectively. The global geometry of fibrations over all 7 double lines can be computed by gluing the local information over each triple point. As all of them have a structure of $\PP^1$ bundle their Betti numbers can be easily computed from the Leray spectral sequence. In the same way we can compute the Betti number of all intersections between components, having in mind that the pinch points on lines $l_3,l_4,l_5,l_6,l_7$ make it so that the conic bundles $Q_i \cap Y$ over these lines have degenerate fibers over pinch points consisting of two intersecting lines. 

Finally components and their Betti numbers are as follows.
\begin{itemize}
    \item $Y$ - smooth Calabi-Yau: $(1, 0, 54, 2, 54, 0, 1)$.
    \item $Q_1$ - quadric bundle with two degenerate fibers, each of them having four components: $(1, 0, 9, 0, 9, 0, 1)$.
    \item $Q_2$ - quadric bundle with a single degenerate fiber which has four components: $(1, 0, 6, 0, 6, 0, 1)$.
    \item $Q_3,...,Q_7$ - quadric bundles with two singular fibers: $(1, 0, 2, 0, 2, 0, 1)$.
\end{itemize}

All non-empty intersections between these components are the following.

\begin{itemize}
  \item $Q_{1} \cap Y$ - conic bundle with two degenerate fibers, each of them having three components:$(1,0,6,0,1)$.
  \item $Q_{2} \cap Y$ - conic bundle with a single degenerate fiber which has three components:$(1,0,4,0,1)$.
  \item $Q_{3} \cap Y,...,Q_{7} \cap Y$ - conic bundles with two degenerate fiber, each of them having two components.  :$(1,0,4,0,1)$
  \item $Q_{1} \cap Q_{2} , \, Q_{1} \cap Q_{3}, \, Q_{1} \cap Q_{4}, \, Q_{1} \cap Q_{5}, \, Q_{2} \cap Q_{6}, \, Q_{2} \cap Q_{7}$ - smooth quadrics: $(1,0,2,0,1)$.
  \item $Q_{i} \cap Q_{j} \cap Y$, where $i,j$ are such that $Q_i\ \cap Q_j \neq \emptyset$ - smooth conics: $(1,0,1)$.

\end{itemize}

We obtain the following spectral sequence, each term is presented as a sum of cohomology groups of the appropriate intersections.

\[\begin{tikzcd}[row sep=small]
    H^2(S^{[3]}) & H^4(S^{[2]}) & H^6(S^{[1]}) & 0 & 0\\
    H^1(S^{[3]}) & H^3(S^{[2]}) & H^5(S^{[1]}) & 0 & 0\\
    H^0(S^{[3]}) & H^2(S^{[2]}) & H^4(S^{[1]}) \oplus H^2(S^{[3]}) & H^4(S^{[2]}) & 0\\
    0 & H^1(S^{[2]}) & H^3(S^{[1]}) \oplus H^1(S^{[3]}) & H^3(S^{[2]}) & 0\\
    0 & H^0(S^{[2]}) & H^2(S^{[1]}) \oplus H^0(S^{[3]}) & H^2(S^{[2]}) & H^2(S^{[3]}) \\
    0 & 0 & H^1(S^{[1]}) & H^1(S^{[2]}) & H^1(S^{[3]})\\
    0 & 0 & H^0(S^{[1]}) & H^0(S^{[2]}) & H^0(S^{[3]})\\
        \arrow[from=1-1, to=1-2]
	\arrow[from=1-2, to=1-3]
	\arrow[from=1-3, to=1-4]
	\arrow[from=1-4, to=1-5]
	\arrow[from=2-1, to=2-2]
	\arrow[from=2-2, to=2-3]
	\arrow[from=2-3, to=2-4]
	\arrow[from=2-4, to=2-5]
	\arrow[from=3-1, to=3-2]
	\arrow[from=3-2, to=3-3]
	\arrow[from=3-3, to=3-4]
	\arrow[from=3-4, to=3-5]
	\arrow[from=4-1, to=4-2]
	\arrow[from=4-2, to=4-3]
	\arrow[from=4-3, to=4-4]
	\arrow[from=4-4, to=4-5]
	\arrow[from=5-1, to=5-2]
	\arrow[from=5-2, to=5-3]
	\arrow[from=5-3, to=5-4]
	\arrow[from=5-4, to=5-5]
	\arrow[from=6-1, to=6-2]
	\arrow[from=6-2, to=6-3]
	\arrow[from=6-3, to=6-4]
	\arrow[from=6-4, to=6-5]
	\arrow[from=7-1, to=7-2]
	\arrow[from=7-2, to=7-3]
	\arrow[from=7-3, to=7-4]
	\arrow[from=7-4, to=7-5]
\end{tikzcd}\]

These vector spaces have the following dimensions.

\[\begin{tikzcd}[row sep=small]
    \CC^6 & \CC^{13} & \CC^8 & 0 & 0\\
    0 & 0 & 0 & 0 & 0\\
    \CC^6 & \CC^{42} & \CC^{85} & \CC^{13} & 0\\
    0 & 0 & \CC^2 & 0 & 0\\
    0 & \CC^{13} & \CC^{85} & \CC^{42} & \CC^6 \\
    0 & 0 & 0 & 0 & 0\\
    0 & 0 & \CC^8 & \CC^{13} & \CC^6 \\
        \arrow[from=1-1, to=1-2]
	\arrow[from=1-2, to=1-3]
	\arrow[from=1-3, to=1-4]
	\arrow[from=1-4, to=1-5]
	\arrow[from=2-1, to=2-2]
	\arrow[from=2-2, to=2-3]
	\arrow[from=2-3, to=2-4]
	\arrow[from=2-4, to=2-5]
	\arrow[from=3-1, to=3-2]
	\arrow[from=3-2, to=3-3]
	\arrow[from=3-3, to=3-4]
	\arrow[from=3-4, to=3-5]
	\arrow[from=4-1, to=4-2]
	\arrow[from=4-2, to=4-3]
	\arrow[from=4-3, to=4-4]
	\arrow[from=4-4, to=4-5]
	\arrow[from=5-1, to=5-2]
	\arrow[from=5-2, to=5-3]
	\arrow[from=5-3, to=5-4]
	\arrow[from=5-4, to=5-5]
	\arrow[from=6-1, to=6-2]
	\arrow[from=6-2, to=6-3]
	\arrow[from=6-3, to=6-4]
	\arrow[from=6-4, to=6-5]
	\arrow[from=7-1, to=7-2]
	\arrow[from=7-2, to=7-3]
	\arrow[from=7-3, to=7-4]
	\arrow[from=7-4, to=7-5]
\end{tikzcd}\]

The maps from the leftmost column come from inclusion of $\PP^1$ into disjoint components and are therefore injective. By duality, the maps to the rightmost column are surjective. The dimension of the image of the map $H^4(S^{[2]}) \rr H^6(S^{[1]})$ is $7$, which can be explicitly computed from incidence relations.

The only maps left are $H^2(S^{[2]}) \rr H^4(S^{[1]}) \oplus H^2(S^{[3]})$ in the second row, and the dual map $H^2(S^{[1]}) \oplus H^0(S^{[3]}) \rr H^2(S^{[2]})$ in the fourth row.

We compute the dimension of the kernel of the map $\phi \colon H^2(S^{[2]}) \rr H^4(S^{[1]}) \oplus H^2(S^{[3]})$, which in this case is $\phi \colon \CC^{42} \rr \CC^{85}$. As there are no quadruple intersections, whose cohomologies would map to $H^2(S^{[3]})$, the image of $\phi$ is contained in $H^4(S^{[1]})$. Poincaré duality yields that $\phi$ corresponds to the map $\phi^* \colon H_2(S^{[1]}) \rr H_2(S^{[2]})$. Let us first precisely describe generators of $H_2(S^{[2]})$ and $H_2(S^{[1]})$. As signs of maps in the monodromy weight spectral sequence depend on the ordering components let us enumerate all the components by giving the smooth Calabi-Yau threefold $Y$ number 0, and quadrics $Q_i$ the numbers that appear as their indices.

We list generators of homology group $H_2(S^{[1]})$ for components other then the Calabi-Yau threefold $Y$ and all generators of the homology group $H_2(S^{[2]})$.

\begin{itemize}
    \item For $Q_1$: $e^1_1$ - smooth section, $e^1_2,e^1_3,e^1_4,e^1_5,e^1_6,e^1_7$ pairs of rulings on quadrics in special fibers lying on the intersection with components $Q_2,Q_3,Q_4$ respectively, $e^1_8,e^1_9$ lines connecting quadrics in special fibres.

    \item For $Q_2$: $e^2_1$ - smooth section, $e^2_2,e^2_3,e^2_4,e^2_5$ pairs of rulings on quadrics in the special fibers lying on the intersection with components $Q_6,Q_7$ respectively, $e^2_6$ line connecting quadrics in the special fibre (which was denoted on the previous diagram as $C \cap Y$).

    \item for $Q_3,Q_4,Q_5,Q_6,Q_7$: $e^i_1$ - ruling on a generic fiber, $e^i_2$ - smooth section,
    \item for $Y \cap Q_1$: $e^{01}_1,e^{01}_2,e^{01}_3,e^{01}_4$ - smooth conics in the intersection with components $Q_2,Q_3,Q_4,Q_5$ respectively, $e^{01}_5$ - smooth section, $e^{01}_6$ line connecting conics $e^{01}_1$ and $e^{01}_2$,
    \item for $Y \cap Q_2$: $e^{02}_1,e^{02}_2$ - smooth conics on the intersection with components $Q_6,Q_7$, $e^{02}_3$ - smooth section, $e^{02}_4$ line connecting conics $e^{02}_1$ and $e^{02}_2$,
    \item for $Y \cap Q_i, i = 3,4,5,6,7$: $e^{0i}_1,e^{0i}_2$ - lines on special fibers over pinch points, $e^{0i}_3$ - generic fiber, $e^{0i}_4$ - smooth section,
    \item for other $Q_i \cap Q_j$, $e^{ij}_1,e^{ij}_2$ - two rulings on a quadric.
\end{itemize}

The listed components generate second homology group of each summand, which follows directly from computing the intersection matrices of these curves and appropriate divisors. Curves generating the second homology of Calabi-Yau threefolds come from sections of exceptional divisors from each blow-up, there is also a single curve coming from a lift of generic line on $\PP^3$.

To obtain the matrix describing $\phi$ we compute the image of the basis of $H^2(S^{[2]})$ in the coordinates given by the basis of $H^2(S^{[1]})$.

Sections of conic bundles are mapped one-to-one to the sections of appropriate quadric bundles. Moreover, nothing else is mapped onto sections of quadric bundles, therefore kernel of $\phi$ is contained in the space generated by other elements.

Each quadric in $S^{[1]}$ is the intersection of two fibrations. On one of these fibrations it is a generic fiber and on the other it is a component of a special fibre. For rulings on quadrics $Q_1 \cap Q_3, Q_1 \cap Q_4, Q_2 \cap Q_6,Q_2 \cap Q_7$, we have curves in the basis of $H^2(S{[0]})$ on which they precisely map to. Rulings on the quadric $Q_1 \cap Q_2$ get mapped onto rulings on the component in a special fibre of the bundle $Q_1$ and on rulings on a generic fiber in the bundle $Q_2$. By deforming a generic fiber of the bundle $Q_2$ we get that these rulings are homologically equivalent to curves $e^2_2 + e^2_4 +e^2_6$ and $e^2_3 + e^2_5+e^2_6$ respectively. Rulings on the quadric $Q_1 \cap Q_5$ get mapped to rulings on a generic fiber of $Q_5$ and rulings on a component of a special fibre of $Q_1$. Let us call the image of these rulings in the bundle $Q_1$ as $r_1,r_2$, by deforming curves on a generic fiber we get the following relations $r_1+e^1_9+e^1_7 = e^1_2+e^1_4+e^1_8$, $r_2+e^1_9+e^1_6=e^1_3+e^1_5+e^1_8$.

Generic fibers of conic bundles $Y \cap Q_i, i = 3,4,5,6,7$ are mapped to curves on Calabi-Yau threefold $Y$ which lie on the blow-ups of double lines corresponding to these bundles and to rulings with multiplicity two on generic fibers of the appropriate quadric bundles.

Lines over pinch points in conic bundles $Y \cap Q_i, i = 3,4,5,6,7$ can be mapped onto curves in the Calabi-Yau threefold $Y$ lying on the blow-up of the fourfold points which correspond to these pinch points. Nothing else is mapped onto these curves in Calabi-Yau threefold and therefore they can not be in the kernel of $\phi$.

The curves in $H_2(S^{[1]})$ we need to consider are conics in the singular fibers in conic bundles $Y \cap Q_1$ and $Y \cap Q_2$. Since the image of $H_2(S^{[2]}) \rr H_2(S^{[1]})$ lies in the kernel of $\phi$, we have that the image of each of the curves $e^{01}_1,e^{01}_2,e^{01}_3,e^{01}_4,e^{02}_1,e^{02}_2$ is a linear combination of other curves we described earlier.

The only curves remaining are $e^{01}_6,e^{02}_4$. To each of them we can associate curves on the Calabi-Yau threefold $Y$ corresponding to the blow-ups of double lines corresponding to these bundles.

With these observations we can write the matrix of $\phi^*$ keeping only the relevant rows and columns which do not belong to the image of .

\setcounter{MaxMatrixCols}{30}

$\begin{pNiceMatrix}[first-row,first-col]
    & e^1_2 & e^1_3 & e^1_4 & e^1_5 & e^1_6 & e^1_7 & e^1_8 & e^1_9 & e^2_2 & e^2_3 & e^2_4 & e^2_5 & e^2_6 & e^3_1 & e^4_1 & 5^6_1 & e^6_1 & e^7_1 &    \\[3pt]
e^{12}_1 & 1&  &  &  &  &  &  &  &-1&  &-1&  &-1&  &  &  &  &  &\\[3pt]
e^{12}_2 &  & 1&  &  &  &  &  &  &  &-1&  &-1&-1&  &  &  &  &  &\\[3pt]
e^{13}_1 &  &  & 1&  &  &  &  &  &  &  &  &  &  &-1&  &  &  &  &\\[3pt]
e^{13}_2 &  &  &  & 1&  &  &  &  &  &  &  &  &  &-1&  &  &  &  &\\[3pt]
e^{14}_1 &  &  &  &  &  & 1&  &  &  &  &  &  &  &  &-1&  &  &  &\\[3pt]
e^{14}_2 &  &  &  &  &  &  & 1&  &  &  &  &  &  &  &-1&  &  &  &\\[3pt]
e^{15}_1 & 1&  & 1&  & 1&-1&  &-1&  &  &  &  &  &  &  &-1&  &  &\\[3pt]
e^{15}_2 &  & 1&  & 1& 1&  &-1&-1&  &  &  &  &  &  &  &-1&  &  &\\[3pt]
e^{26}_1 &  &  &  &  &  &  &  &  & 1&  &  &  &  &  &  &  &-1&  &\\[3pt]
e^{26}_2 &  &  &  &  &  &  &  &  &  & 1&  &  &  &  &  &  &-1&  &\\[3pt]
e^{27}_1 &  &  &  &  &  &  &  &  &  &  & 1&  &  &  &  &  &  &-1&\\[3pt]
e^{27}_2 &  &  &  &  &  &  &  &  &  &  &  & 1&  &  &  &  &  &-1&\\[3pt]
\end{pNiceMatrix}$

The kernel of this matrix is one dimensional and is spanned by the chain $C = e^{12}_1 - e^{12}_2 + e^{13}_1 - e^{13}_2 + e^{14}_1 - e^{14}_2 + e^{15}_1 - e^{15}_2 + e^{26}_1 - e^{26}_2 + e^{27}_1 - e^{27}_2$. The chain $C$ has two rulings from every quadric, for every quadric one of its rulings appears with a $+1$ and the other one with $-1$ coefficient. We note, that this chain cannot be deformed to a chain in the rigid Calabi-Yau threefold, as its intersection with each of these quadrics is a smooth conic, which is homologically equivalent to taking two rulings each with a $+1$ coefficient. Thus, the second page of the spectral sequence is the following.

\[\begin{tikzcd}[row sep=small]
    0 & 0 & \CC & 0 & 0\\
    0 & 0 & 0 & 0 & 0\\
    0 & \CC & \CC^{37} & 0 & 0\\
    0 & 0 & \CC^2 & 0 & 0\\
    0 & 0 & \CC^{37} & \CC & 0 \\
    0 & 0 & 0 & 0 & 0\\
    0 & 0 & \CC & 0 & 0 \\
        \arrow[from=1-1, to=1-2]
	\arrow[from=1-2, to=1-3]
	\arrow[from=1-3, to=1-4]
	\arrow[from=1-4, to=1-5]
	\arrow[from=2-1, to=2-2]
	\arrow[from=2-2, to=2-3]
	\arrow[from=2-3, to=2-4]
	\arrow[from=2-4, to=2-5]
	\arrow[from=3-1, to=3-2]
	\arrow[from=3-2, to=3-3]
	\arrow[from=3-3, to=3-4]
	\arrow[from=3-4, to=3-5]
	\arrow[from=4-1, to=4-2]
	\arrow[from=4-2, to=4-3]
	\arrow[from=4-3, to=4-4]
	\arrow[from=4-4, to=4-5]
	\arrow[from=5-1, to=5-2]
	\arrow[from=5-2, to=5-3]
	\arrow[from=5-3, to=5-4]
	\arrow[from=5-4, to=5-5]
	\arrow[from=6-1, to=6-2]
	\arrow[from=6-2, to=6-3]
	\arrow[from=6-3, to=6-4]
	\arrow[from=6-4, to=6-5]
	\arrow[from=7-1, to=7-2]
	\arrow[from=7-2, to=7-3]
	\arrow[from=7-3, to=7-4]
	\arrow[from=7-4, to=7-5]
\end{tikzcd}\]

We recover betti numbers of a generic fiber of the family and a mixed Hodge structure on $H^3(S_\infty)$ which is non pure.

\section{Conclusion}

The three examples of degenerations studied in the previous section provide us with two distinct behaviors of the third cohomology group of the degenerate fiber. In the first and last example, the mixed Hodge structure on the third cohomology group was non-pure of the form $\mathbb C\oplus \mathbb C^2 \oplus \mathbb C$  with the two dimensional part coming from cycles on a rigid Calabi-Yau threefold and the one-dimensional ones coming from cycles on double intersections. On the other hand in the example studied in sect. \ref{ex2} the mixed Hodge structure on the third cohomology group is pure.

Singularities of a degenerate fiber $Y_w$ in any one-parameter family of double octics admit a crepant resolution $Y$ (see \cite{CYK}), which is a rigid Calabi-Yau threefold. In this situation we can consider the family $\mathcal Y$ as a smoothing of the contraction $Y_w$ of the degenerate fiber, hence this process gives a geometric transition between the family $\mathcal Y$ and the rigid Calabi-Yau threefold $Y$. The monodromy weight spectral sequence yields relation between invariants of $Y$ and a generic fiber of $\mathcal Y$. Finally, we can compute the mixed Hodge structures on the middle cohomology of $Y$,

A one-parameter family of Calabi-Yau threefolds defines a Variation of Hodge Structures and Fuchs differential operator satisfied by the period integrals,  considered degenerations corresponds to conifold points. There exist however conifold points, where the degenerate configurations of eight planes do not satisfy the definition of octic arrangement. The failure of the octic arrangement condition in this situation is a consequence of the choice of parametrization, the situation can be improved by a change of parametrization.

A one-parameter family of Calabi-Yau threefolds may admit a singular point where the Picard-Fuchs operator has two different initial powers, because of the expected connection with a K3 surfaces singular points of that type were called type K. In the case of a degenerate point of type K, the degenerate configuration of eight planes is not an octic arrangement, hence the method of this paper cannot be used directly. We shall study semistable degeneration at K-point of a one-parameter families of double octics in a separate paper.

The techniques presented in this paper, easily translate to an arithmetic setting, where we study families of schemes over a discrete valuation rings. An example of such construction can be found in \cite{CO}, where similar combinatorial arguments allow construction of an arithmetic semistable degeneration, and a monodromy weight spectral sequence in mixed characteristic gives a concise proof of crystallinity of Galois representations on the third cohomology group
of a Calabi-Yau threefold.

\textbf{Declarations} 

\textbf{Data availability} Purely theoretical approach was used, therefore no datasets were needed.

\textbf{Conflict of interest} The author states that there is no conflict of interest.


\newpage
\appendix
\section{List of local degeneration types}
\label{singularity classification}

We list all combinatorial types of local degenerations that occur in one-parameter families of double octics, for degenerate fibers which have a Calabi-Yau resolution. Assume that $P(x,y,z,t,w)=\prod_{i=1}^{8}P_{i}(x,y,z,t,w)$ is a product of eight polynomials $P_{i}$ linear in variables $x,y,z,t$. Assume moreover that for all values of  $w\in\CC$  with at most finite number of exceptions $w\in\Sigma$, the zero-set 

$$D_{w}:=\{(x,y,z,t)\in \PP^{3}: P(x,y,z,t,w)=0\}$$

is an octic arrangement of the same combinatorial type.

By using the classification of double octic from \cite{CYK} one can easily list all such local degenerations. For a full list of new incidences in degenerate fibers of one-parameter families of octic arrangements we refer to the appendix \ref{full list of degenerations}. 

For the sake of simplicity, we assume that the degenerate fiber, $X_{w}$ is at $w=0$.

\subsection{New $l_{3}$ line}\label{new l3}
The degeneration happens when three planes $D_{i_{1},w}$, $D_{i_{2},w}$ and $D_{i_{3},w}$, $1\le i_{1}<i_{2}<i_{3}\le8$ intersect in one point  for $w\not\in\Sigma$, while $D_{i_{1},0}\cap D_{i_{2},0}\cap D_{i_{3},0}$ is a triple line. Moreover, we assume that for small $w\not=0$ the intersection points $D_{i_{1},w}\cap D_{i_{2},w}\cap D_{i_{3},w}$ do not belong to a neighborhood of the point $(0:0:0:1)$ in $\PP^{3}$.

This degeneration is combinatorially equivalent to 
$$u^2=xy(x+y+w).$$

\subsection{New $p_{4}^{0}$ point}
This degeneration happen when four planes $D_{i_{1},w}$, $D_{i_{2},w}$, $D_{i_{3},w}$, $D_{i_{4},w}$, $1\le i_{1}<i_{2}<i_{3}<i_{4}\le 8$, are in general position for general $w$, while the point $D_{i_{1},0}\cap D_{i_{2},0}\cap D_{i_{3},0}\cap D_{i_{4},0}$ is a $p_{4}^{0}$ point.

This degeneration is combinatorially equivalent to 
$$u^2=xyz(x+y+z+w).$$

\subsection{A $p_{5}^{1}$ point degenerates to a $p_{5}^{2}$ point}
For a generic value of $w$ the planes $D_{i_{1}}$, $D_{i_{2}}$, $D_{i_{3}}$ intersect along a triple line, the planes $D_{i_{4}}$, $D_{i_{5}}$  intersect the triple line at the same point. 
For the special value, the planes $D_{i_{1}}$, $D_{i_{4}}$ and $D_{i_{5}}$ also intersect along a triple line.
This degeneration is combinatorially equivalent to 
$$u^2=xy(x+y)z(x+wy+z)$$

\subsection{Two $p_{4}^{1}$ points collide to a $p_{5}^{2}$ point}

For a generic value of $w$ the planes $D_{i_{1}}$, $D_{i_{2}}$, $D_{i_{3}}$ intersect along a triple line, the planes $D_{i_{4}}$, $D_{i_{5}}$ are generic. For the special value, the planes $D_{i_{1}}$, $D_{i_{4}}$ and $D_{i_{5}}$ also intersect along a triple line. 
The $p_{4}^{1}$ points $D_{i_{1}}\cap D_{i_{2}}\cap D_{i_{3}}\cap D_{i_{4}}$ $D_{i_{1}}\cap D_{i_{2}}\cap D_{i_{3}}\cap D_{i_{4}}$ collide. 

This degeneration is combinatorially equivalent to 
$$u^2=xy(x+y)z(x+z+w)$$

\subsection{Two $p_{4}^{1}$ points collide to a $p_{5}^{1}$ point}

For a generic value of $w$ the planes $D_{i_{1}}$, $D_{i_{2}}$, $D_{i_{3}}$ intersect along a triple line, the planes $D_{i_{4}}$, $D_{i_{5}}$ are generic. For the special value, the planes $D_{i_{1}}$, $D_{i_{4}}$ and $D_{i_{5}}$ also intersect along a triple line. 
The $p_{4}^{1}$ points $D_{i_{1}}\cap D_{i_{2}}\cap D_{i_{3}}\cap D_{i_{4}}$ $D_{i_{1}}\cap D_{i_{2}}\cap D_{i_{3}}\cap D_{i_{4}}$ collide. 

This degeneration is combinatorially equivalent to 
$$u^2=xy(x+y)z(x+y+z+w)$$

\subsection{A $p_{4}^{0}$ point degenerates to a $p_{5}^{2}$ point}For generic value of $w$ the planes $D_{i_{1}}, D_{i_{2}}$, $D_{i_{3}}$ and $D_{i_{4}}$ are general four planes intersecting at a $p_{4}^{0}$ and the intersection $D_{i_{1}} \cap D_{i_{2}} \cap$ $D_{i_{3}} \cap D_{i_{4}} \cap D_{i_{5}}$ is empty, at the special value of $w$ the plane $D_{i_{5}}$ contain lines $D_{i_{1}} \cap D_{i_{2}}$ and $D_{i_{3}} \cap D_{i_{4}}$.

This degeneration is combinatorially equivalent to

$$
u^{2}=x y z(x+y+z)(x+y+w)
$$
\subsection{New $p_{4}^{1}$ point}
For generic $w$ the planes $D_{i_{1}}$, $D_{i_{2}}$, $D_{i_{3}}$ and $D_{i_{4}}$ are in general position, at the special value of $w$ the plane $D_{i_{3}}$ contains the line $D_{i_{1}}\cap D_{i_{2}}$, which become  a triple line.

This degeneration is combinatorially equivalent to 
\[u^2=xy(x+y+w)z\]

\subsection{A $p_{4}^{0}$ point degenerates to a $p_{4}^{1}$ point}
For generic $w$ the planes $D_{i_{1}}$, $D_{i_{2}}$, $D_{i_{3}}$ and $D_{i_{4}}$ are general four planes intersecting at one point, at the special value of $w$ the planes $D_{i_{1}}$, $D_{i_{2}}$, $D_{i_{3}}$
intersect at a triple line.

This degeneration is combinatorially equivalent to 
$$u^2=xy(x+y+zw)z$$

\subsection{A $p_{4}^{0}$ point degenerates to a $p_{5}^{1}$ point}

For generic $w$ the planes $D_{i_{1}}$, $D_{i_{2}}$, $D_{i_{3}}$ and $D_{i_{4}}$ are general four planes intersecting at one point and the plane $D_{i_5}$ does not contain it, at the special value of $w$ the plane $D_{i_{5}}$ contains line $D_{i_{1}}\cap D_{i_{2}}$, which becomes a triple line and does not contain the line $D_{i_{3}}\cap D_{i_{4}}$.

This degeneration is combinatorially equivalent to 
$$u^2=xyz(x+y+z)(x-y+w)$$

\subsection{A $p_{5}^{0}$ point degenerates to a $p_{5}^{2}$ point}

For generic $w$ the planes $D_{i_{1}}$, $D_{i_{2}}$, $D_{i_{3}}$, $D_{i_{4}}$ and $D_{i_{5}}$ are general planes intersecting at a $p_{4}^{0}$ point, at the special value of $w$ the plane $D_{i_{5}}$ contains lines $D_{i_{1}}\cap D_{i_{2}}$ and $D_{i_{3}}\cap D_{i_{4}}$, which become triple lines.

This degeneration is combinatorially equivalent to 
$$u^2=xyz(x+y+wz)(x+wy+z)$$

\subsection{A $p_{5}^{0}$ point degenerates to a $p_{5}^{1}$ point}

For generic $w$ the planes $D_{i_{1}}$, $D_{i_{2}}$, $D_{i_{3}}$, $D_{i_{4}}$ and $D_{i_{5}}$ are general planes intersecting at one point, at the special value of $w$ the plane $D_{i_{5}}$ contains the line $D_{i_{1}}\cap D_{i_{2}}$ and does not contain $D_{i_3} \cap D{i_4}$.

This degeneration is combinatorially equivalent to 
$$u^2=xyz(x+y+wz)(x+2y+z)$$

\section{List of double octics degenerations}
\label{full list of degenerations}

Below we present a complete list of degenerations of one parameter families of octic arrangements, where degenerate fiber is also an octic arrangement. Each row represents a single degeneration and the column are organized in the following way.

\begin{itemize}
    \item The first column is the number of a family from the classification in \cite{CYB}.
    \item The second column represents the value of the parameter for which the degeneration occurs.
    \item The third column is a number of a rigid Calabi-Yau corresponding to the octic arrangement in the degenerate fiber from the classification \cite{CYB}.
    \item The fourth column lists which local degenerations from appendix \ref{singularity classification} appear in the given degeneration.
\end{itemize}

\def\arraystretch{1.2}
\begin{longtable}{|c|c|c|l|}
	\caption{Double octics degeneration list}	\\
\hline \parbox{10mm}{\rule{0mm}{5mm}Gen.\\\rule[-2mm]{0mm}{6mm} Arr.} &Special point&\parbox{10mm}{\rule{0mm}{5mm}Spec.\\\rule[-2mm]{0mm}{6mm} Arr.}& Degeneration types \\[1mm]
\hline \endhead\hline 
\endfoot\hline 
2 & 1 & 1 & 2 [ 3, 5, 7, 8 ]\\
5 & 1 & 3 & 2 [ 3, 4, 7, 8 ]\\
5 & -1 & 3 & 2 [ 3, 5, 6, 8 ]\\
8 & -1 & 1 & 1 [ 5, 6, 8 ], 3 [ 1, 4, 5, 6, 8 ], 4 [ 2, 5, 6, 7, 8 ], 7 [ 3, 5, 6, 8 ]\\
10 & 1 & 1 & 1 [ 5, 7, 8 ], 4 [ 1, 4, 5, 7, 8 ], 4 [ 2, 5, 6, 7, 8 ], 8 [ 3, 5, 7, 8 ]\\
16 & 0 & 1 & 1 [ 2, 6, 7 ], 1 [ 4, 6, 8 ], 3 [ 1, 2, 3, 6, 7 ], 3 [ 1, 4, 5, 6, 8 ], \\
&&& 6 [ 2, 4, 6, 7, 8 ], 7 [ 2, 5, 6, 7 ], 7 [ 3, 4, 6, 8 ]\\
16 & $\infty$ & 1 & 1 [ 3, 6, 7 ], 1 [ 5, 6, 8 ], 3 [ 1, 2, 3, 6, 7 ], 3 [ 1, 4, 5, 6, 8 ], \\
&&&  6 [ 3, 5, 6, 7, 8 ], 7 [ 2, 5, 6, 8 ], 7 [ 3, 4, 6, 7 ]\\
20 & $\infty$ & 1 & 1 [ 3, 6, 7 ], 1 [ 4, 6, 8 ], 3 [ 1, 2, 3, 6, 7 ], 4 [ 1, 4, 5, 6, 8 ],  \\
&&& 6 [ 3, 4, 6, 7, 8 ], 7 [ 3, 5, 6, 7 ], 8 [ 2, 4, 6, 8 ]\\
20 & 1 & 19 & 2 [ 3, 5, 6, 8 ]\\
20 & -1 & 3 & 1 [ 1, 6, 7 ], 3 [ 1, 2, 3, 6, 7 ], 4 [ 1, 4, 5, 6, 7 ],  7 [ 1, 6, 7, 8 ]\\
33 & -1 & 3 & 1 [ 1, 7, 8 ], 4 [ 1, 2, 3, 7, 8 ], 4 [ 1, 4, 5, 7, 8 ],  8 [ 1, 6, 7, 8 ]\\
33 & -2 & 32 & 2 [ 3, 5, 7, 8 ]\\
34 & 1 & 19 & 5 [ 1, 2, 3, 7, 8 ]\\
34 & -1 & 19 & 5 [ 1, 4, 5, 7, 8 ]\\
35 & 0 & 1 & 1 [ 2, 6, 7 ], 1 [ 5, 7, 8 ], 4 [ 1, 2, 3, 6, 7 ], 4 [ 1, 4, 5, 7, 8 ],  \\
&&& 6 [ 2, 5, 6, 7, 8 ], 8 [ 2, 4, 6, 7 ],  8 [ 3, 5, 7, 8 ]\\
35 & $\infty$ & 1 & 1 [ 3, 7, 8 ], 1 [ 4, 6, 7 ], 4 [ 1, 2, 3, 7, 8 ],  4 [ 1, 4, 5, 6, 7 ], \\
&&& 6 [ 3, 4, 6, 7, 8 ], 8 [ 2, 4, 6, 7 ], 8 [ 3, 5, 7, 8 ]\\
35 & -1 & 32 & 2 [ 1, 6, 7, 8 ]\\
36 & -1 & 32 & 2 [ 3, 5, 6, 7 ]\\
70 & $\infty$ & 3 & 1 [ 3, 4, 5 ], 1 [ 3, 6, 7 ], 3 [ 1, 2, 3, 4, 5 ],  3 [ 1, 2, 3, 6, 7 ], \\
&&& 6 [ 3, 4, 5, 6, 7 ], 7 [ 3, 4, 5, 8 ],  7 [ 3, 6, 7, 8 ]\\
70 & -1 & 69 & 2 [ 3, 5, 6, 8 ]\\
71 & $\infty$ & 1 & 1 [ 2, 4, 5 ], 1 [ 3, 6, 7 ], 1 [ 5, 7, 8 ], 3 [ 1, 2, 3, 4, 5 ],  \\
&&& 3 [ 1, 2, 3, 6, 7 ], 6 [ 2, 4, 5, 7, 8 ], 6 [ 3, 5, 6, 7, 8 ],  \\
&&& 7 [ 2, 4, 5, 6 ], 7 [ 3, 4, 6, 7 ], 8 [ 1, 5, 7, 8 ]\\
71 & 1 & 1 & 1 [ 2, 6, 7 ], 1 [ 3, 4, 5 ], 1 [ 4, 6, 8 ], 3 [ 1, 2, 3, 4, 5 ],  \\
&&& 3 [ 1, 2, 3, 6, 7 ], 6 [ 2, 4, 6, 7, 8 ], 6 [ 3, 4, 5, 6, 8 ],  \\
&&& 7 [ 2, 5, 6, 7 ], 7 [ 3, 4, 5, 7 ], 8 [ 1, 4, 6, 8 ]\\
71 & 2 & 69 & 2 [ 4, 5, 6, 7 ]\\
72 & 1 & 19 & 1 [ 3, 6, 7 ], 3 [ 1, 2, 3, 6, 7 ], 7 [ 3, 4, 6, 7 ],  \\
&&& 7 [ 3, 5, 6, 7 ], 7 [ 3, 6, 7, 8 ]\\
72 & -1 & 19 & 1 [ 3, 4, 5 ], 3 [ 1, 2, 3, 4, 5 ], 7 [ 3, 4, 5, 6 ],  \\
&&& 7 [ 3, 4, 5, 7 ], 7 [ 3, 4, 5, 8 ]\\
73 & -1 & 69 & 2 [ 1, 5, 7, 8 ]\\
94 & 1 & 1 & 1 [ 1, 6, 7 ], 1 [ 3, 4, 5 ], 1 [ 5, 7, 8 ], 3 [ 1, 2, 3, 4, 5 ],  \\
&&& 4 [ 1, 2, 3, 6, 7 ], 6 [ 1, 5, 6, 7, 8 ], 6 [ 3, 4, 5, 7, 8 ],  \\
&&& 7 [ 3, 4, 5, 6 ], 8 [ 1, 4, 6, 7 ], 8 [ 2, 5, 7, 8 ]\\
94 & $\frac12(1-\sqrt{-3})$ & 451 & 2 [ 3, 5, 6, 7 ]\\
94 & $\frac12(1+\sqrt{-3})$ & 451 & 2 [ 3, 5, 6, 7 ]\\
95 & 1 & 93 & 2 [ 2, 5, 6, 7 ]\\
95 & -1 & 3 & 1 [ 3, 4, 5 ], 1 [ 3, 7, 8 ], 3 [ 1, 2, 3, 4, 5 ], 4 [ 1, 2, 3, 7, 8 ],  \\
&&& 6 [ 3, 4, 5, 7, 8 ], 7 [ 3, 4, 5, 6 ], 8 [ 3, 6, 7, 8 ]\\
96 & 2 & 32 & 1 [ 3, 4, 5 ], 3 [ 1, 2, 3, 4, 5 ], 7 [ 3, 4, 5, 6 ],  \\
&&& 7 [ 3, 4, 5, 7 ], 7 [ 3, 4, 5, 8 ]\\
97 & 1/2 & 93 & 2 [ 4, 5, 7, 8 ]\\
98 & 1 & 93 & 2 [ 3, 6, 7, 8 ]\\
98 & -1 & 19 & 1 [ 3, 4, 5 ], 3 [ 1, 2, 3, 4, 5 ], 7 [ 3, 4, 5, 6 ], 9 [ 3, 4, 5, 7, 8 ]\\
99 & 1 & 19 & 1 [ 3, 7, 8 ], 4 [ 1, 2, 3, 7, 8 ], 7 [ 3, 5, 7, 8 ], 7 [ 3, 6, 7, 8 ],  \\
&&& 8 [ 3, 4, 7, 8 ]\\
99 & -1 & 19 & 1 [ 3, 4, 5 ], 3 [ 1, 2, 3, 4, 5 ], 7 [ 3, 4, 5, 6 ], 9 [ 3, 4, 5, 7, 8 ]\\
100 & 1 & 69 & 5 [ 1, 2, 3, 7, 8 ]\\
144 & 0 & 19 & 1 [ 1, 5, 6 ], 11 [ 1, 4, 5, 6, 7 ], 4 [ 1, 2, 3, 5, 6 ], 7 [ 1, 5, 6, 8 ]\\
144 & 2 & 19 & 1 [ 1, 4, 7 ], 11 [ 1, 4, 5, 6, 7 ], 4 [ 1, 2, 3, 4, 7 ], 7 [ 1, 4, 7, 8 ]\\
152 & 0 & 32 & 1 [ 2, 6, 8 ], 4 [ 1, 2, 3, 6, 8 ], 7 [ 2, 5, 6, 8 ], 7 [ 2, 6, 7, 8 ],  \\
&&& 8 [ 2, 4, 6, 8 ]\\
152 & 1 & 19 & 1 [ 3, 5, 8 ], 4 [ 1, 2, 3, 5, 8 ], 8 [ 3, 4, 5, 8 ], 9 [ 3, 5, 6, 7, 8 ]\\
152 & 2 & 32 & 1 [ 1, 7, 8 ], 4 [ 1, 2, 3, 7, 8 ], 7 [ 1, 5, 7, 8 ], 7 [ 1, 6, 7, 8 ],  \\
&&& 8 [ 1, 4, 7, 8 ]\\
153 & $\infty$ & 3 & 1 [ 1, 5, 6 ], 1 [ 1, 7, 8 ], 4 [ 1, 2, 3, 5, 6 ], 4 [ 1, 2, 3, 7, 8 ],  \\
&&& 6 [ 1, 5, 6, 7, 8 ], 8 [ 1, 4, 5, 6 ], 8 [ 1, 4, 7, 8 ]\\
153 & 1 & 19 & 1 [ 3, 5, 8 ], 4 [ 1, 2, 3, 5, 8 ], 8 [ 3, 4, 5, 8 ], 9 [ 3, 5, 6, 7, 8 ]\\
153 & -1 & 93 & 5 [ 1, 2, 3, 6, 7 ]\\
154 & 0 & 1 & 1 [ 1, 4, 6 ], 1 [ 2, 5, 7 ], 1 [ 6, 7, 8 ], 4 [ 1, 2, 3, 4, 6 ],  \\
&&& 4 [ 1, 2, 3, 5, 7 ], 6 [ 1, 4, 6, 7, 8 ], 6 [ 2, 5, 6, 7, 8 ],  \\
&&& 8 [ 1, 4, 5, 6 ], 8 [ 2, 4, 5, 7 ], 8 [ 3, 6, 7, 8 ]\\
154 & -1 & 32 & 1 [ 3, 6, 7 ], 4 [ 1, 2, 3, 6, 7 ], 7 [ 3, 4, 6, 7 ], 7 [ 3, 5, 6, 7 ],  \\
&&& 8 [ 3, 6, 7, 8 ]\\
155 & $\frac12(1-\sqrt{-3})$ & 451 & 5 [ 1, 2, 3, 6, 7 ]\\
155 & $\frac12(1+\sqrt{-3})$ & 451 & 5 [ 1, 2, 3, 6, 7 ]\\
197 & 0 & 19 & 1 [ 1, 2, 5 ], 1 [ 5, 6, 7 ], 11 [ 1, 2, 3, 4, 5 ], 6 [ 1, 2, 5, 6, 7 ],  \\
&&& 7 [ 1, 2, 5, 8 ], 7 [ 3, 5, 6, 7 ], 7 [ 4, 5, 6, 7 ], 8 [ 5, 6, 7, 8 ]\\
197 & $\infty$ & 3 & 1 [ 1, 3, 5 ], 1 [ 2, 4, 5 ], 1 [ 5, 6, 8 ], 10 [ 1, 2, 3, 4, 5 ],  \\
&&& 6 [ 1, 3, 5, 6, 8 ], 6 [ 2, 4, 5, 6, 8 ], 7 [ 1, 3, 5, 7 ],  \\
&&& 7 [ 2, 4, 5, 7 ], 8 [ 5, 6, 7, 8 ]\\
197 & 1 & 3 & 1 [ 1, 4, 5 ], 1 [ 2, 3, 5 ], 1 [ 5, 7, 8 ], 10 [ 1, 2, 3, 4, 5 ],  \\
&&& 6 [ 1, 4, 5, 7, 8 ], 6 [ 2, 3, 5, 7, 8 ], 7 [ 1, 4, 5, 6 ],  \\
&&& 7 [ 2, 3, 5, 6 ], 8 [ 5, 6, 7, 8 ]\\
197 & 2 & 93 & 1 [ 3, 4, 5 ], 11 [ 1, 2, 3, 4, 5 ], 7 [ 3, 4, 5, 6 ], 7 [ 3, 4, 5, 7 ],  \\
&&& 7 [ 3, 4, 5, 8 ]\\
198 & 0 & 69 & 1 [ 1, 4, 5 ], 11 [ 1, 2, 3, 4, 5 ], 7 [ 1, 4, 5, 6 ], 9 [ 1, 4, 5, 7, 8 ]\\
198 & 1 & 19 & 1 [ 1, 2, 5 ], 1 [ 5, 6, 7 ], 11 [ 1, 2, 3, 4, 5 ], 6 [ 1, 2, 5, 6, 7 ],  \\
&&& 7 [ 1, 2, 5, 8 ], 7 [ 4, 5, 6, 7 ], 7 [ 5, 6, 7, 8 ], 8 [ 3, 5, 6, 7 ]\\
198 & 2 & 69 & 1 [ 2, 4, 5 ], 11 [ 1, 2, 3, 4, 5 ], 7 [ 2, 4, 5, 7 ], 9 [ 2, 4, 5, 6, 8 ]\\
199 & $\infty$ & 1 & 1 [ 1, 3, 5 ], 1 [ 2, 3, 4 ], 1 [ 4, 7, 8 ], 1 [ 5, 6, 8 ],  \\
&&& 6 [ 1, 3, 5, 6, 8 ], 6 [ 2, 3, 4, 7, 8 ], 6 [ 4, 5, 6, 7, 8 ],  \\
&&& 7 [ 1, 3, 5, 7 ], 7 [ 2, 3, 4, 6 ], 8 [ 1, 4, 7, 8 ], 8 [ 2, 5, 6, 8 ],  \\
&&& 10 [ 1, 2, 3, 4, 5 ]\\
199 & -1 & 69 & 1 [ 3, 4, 5 ], 11 [ 1, 2, 3, 4, 5 ], 7 [ 3, 4, 5, 8 ], 9 [ 3, 4, 5, 6, 7 ]\\
200 & $\frac12(1-\sqrt{-3})$ & 451 & 1 [ 2, 3, 4 ], 11 [ 1, 2, 3, 4, 5 ], 7 [ 2, 3, 4, 6 ],  7 [ 2, 3, 4, 7 ], \\
&&& 7 [ 2, 3, 4, 8 ]\\
200 & $\frac12(1+\sqrt{-3})$ & 451 & 1 [ 2, 3, 4 ], 11 [ 1, 2, 3, 4, 5 ], 7 [ 2, 3, 4, 6 ], 7 [ 2, 3, 4, 7 ],  \\
&&& 7 [ 2, 3, 4, 8 ]\\
242 & 1 & 238 & 2 [ 1, 4, 5, 8 ], 2 [ 2, 3, 6, 7 ]\\
242 & -1 & 238 & 2 [ 1, 4, 6, 7 ], 2 [ 2, 3, 5, 8 ]\\
243 & $\infty$ & 3 & 1 [ 2, 7, 8 ], 1 [ 3, 6, 8 ], 1 [ 4, 5, 8 ], 6 [ 2, 3, 6, 7, 8 ],  \\
&&& 6 [ 2, 4, 5, 7, 8 ], 6 [ 3, 4, 5, 6, 8 ], 8 [ 1, 2, 7, 8 ],  \\
&&& 8 [ 1, 3, 6, 8 ], 8 [ 1, 4, 5, 8 ]\\
243 & 1 & 239 & 2 [ 2, 3, 5, 8 ]\\
243 & -1 & 238 & 2 [ 2, 4, 6, 8 ], 2 [ 3, 4, 7, 8 ], 2 [ 5, 6, 7, 8 ]\\
243 & -2 & 240 & 2 [ 4, 6, 7, 8 ]\\
244 & -1/2 & 240 & 2 [ 1, 4, 6, 7 ]\\
244 & -2 & 240 & 2 [ 2, 3, 6, 7 ]\\
246 & 0 & 1 & 1 [ 1, 3, 4 ], 1 [ 2, 5, 6 ], 1 [ 3, 6, 8 ], 1 [ 4, 5, 7 ],  \\
&&& 6 [ 1, 3, 4, 5, 7 ], 6 [ 1, 3, 4, 6, 8 ], 6 [ 2, 3, 5, 6, 8 ],  \\
&&& 6 [ 2, 4, 5, 6, 7 ], 8 [ 1, 2, 3, 4 ], 8 [ 1, 2, 5, 6 ],  \\
&&& 8 [ 3, 6, 7, 8 ], 8 [ 4, 5, 7, 8 ]\\
246 & $\infty$ & 1 & 1 [ 1, 5, 6 ], 1 [ 2, 3, 4 ], 1 [ 3, 6, 7 ], 1 [ 4, 5, 8 ],  \\
&&& 6 [ 1, 3, 5, 6, 7 ], 6 [ 1, 4, 5, 6, 8 ], 6 [ 2, 3, 4, 5, 8 ],  \\
&&& 6 [ 2, 3, 4, 6, 7 ], 8 [ 1, 2, 3, 4 ], 8 [ 1, 2, 5, 6 ], 8 [ 3, 6, 7, 8 ],  \\
&&& 8 [ 4, 5, 7, 8 ]\\
246 & 1 & 241 & 2 [ 3, 4, 5, 6 ]\\
247 & $\infty$ & 93 & 1 [ 2, 3, 4 ], 7 [ 2, 3, 4, 5 ], 7 [ 2, 3, 4, 8 ], 8 [ 1, 2, 3, 4 ],  \\
&&& 9 [ 2, 3, 4, 6, 7 ]\\
247 & -1 & 238 & 2 [ 2, 4, 5, 7 ], 2 [ 2, 4, 6, 8 ], 2 [ 3, 4, 5, 6 ],  2 [ 3, 4, 7, 8 ]\\
247 & -1/2 & 93 & 1 [ 4, 5, 8 ], 7 [ 2, 4, 5, 8 ], 7 [ 3, 4, 5, 8 ], 8 [ 1, 4, 5, 8 ],  \\
&&& 9 [ 4, 5, 6, 7, 8 ]\\
248 & $\infty$ & 19 & 1 [ 2, 3, 4 ], 1 [ 4, 5, 8 ], 6 [ 2, 3, 4, 5, 8 ], 7 [ 4, 5, 6, 8 ],  \\
&&& 7 [ 4, 5, 7, 8 ], 8 [ 1, 2, 3, 4 ], 8 [ 1, 4, 5, 8 ], 9 [ 2, 3, 4, 6, 7 ]\\
248 & 1/2 & 245 & 2 [ 4, 5, 6, 7 ]\\
248 & 1 & 239 & 2 [ 2, 4, 5, 7 ], 2 [ 3, 4, 5, 6 ]\\
248 & -1 & 239 & 2 [ 2, 4, 6, 8 ], 2 [ 3, 4, 7, 8 ]\\
248 & -1/2 & 245 & 2 [ 4, 6, 7, 8 ]\\
249 & 1 & 241 & 2 [ 3, 6, 7, 8 ], 2 [ 4, 5, 7, 8 ]\\
250 & $\infty$ & 69 & 1 [ 1, 3, 4 ], 8 [ 1, 2, 3, 4 ], 9 [ 1, 3, 4, 5, 7 ], 9 [ 1, 3, 4, 6, 8 ]\\
250 & 1/2 & 240 & 2 [ 3, 4, 5, 6 ], 2 [ 3, 4, 7, 8 ]\\
250 & 1 & 93 & 1 [ 4, 6, 7 ], 7 [ 1, 4, 6, 7 ], 7 [ 3, 4, 6, 7 ], 8 [ 2, 4, 6, 7 ], \\
&&& 9 [ 4, 5, 6, 7, 8 ]\\
250 & -1 & 245 & 2 [ 1, 4, 5, 8 ]\\
251 & 0 & 1 & 1 [ 1, 3, 4 ], 1 [ 2, 7, 8 ], 1 [ 3, 5, 7 ], 1 [ 4, 6, 8 ], \\
&&& 6 [ 1, 3, 4, 5, 7 ], 6 [ 1, 3, 4, 6, 8 ], 6 [ 2, 3, 5, 7, 8 ], \\
&&& 6 [ 2, 4, 6, 7, 8 ], 7 [ 3, 5, 6, 7 ], 8 [ 1, 2, 3, 4 ], 8 [ 1, 2, 7, 8 ], \\
&&& 8 [ 4, 5, 6, 8 ]\\
251 & $\infty$ & 19 & 1 [ 2, 3, 4 ], 1 [ 4, 5, 8 ], 6 [ 2, 3, 4, 5, 8 ], 7 [ 1, 4, 5, 8 ], \\
&&& 7 [ 4, 5, 7, 8 ], 8 [ 1, 2, 3, 4 ], 8 [ 4, 5, 6, 8 ], 9 [ 2, 3, 4, 6, 7 ]\\
251 & -1/2 & 93 & 1 [ 4, 6, 7 ], 7 [ 1, 4, 6, 7 ], 7 [ 3, 4, 6, 7 ], 8 [ 2, 4, 6, 7 ], \\
&&& 9 [ 4, 5, 6, 7, 8 ]\\
251 & $\frac12(1-\sqrt{5})$ & 453 & 2 [ 3, 4, 7, 8 ]\\
251 & $\frac12(1+\sqrt{5})$ & 453 & 2 [ 3, 4, 7, 8 ]\\
252 & $\infty$ & 32 & 1 [ 2, 3, 4 ], 1 [ 2, 5, 6 ], 6 [ 2, 3, 4, 5, 6 ], 7 [ 2, 3, 4, 7 ], \\
&&& 7 [ 2, 3, 4, 8 ], 7 [ 2, 5, 6, 7 ], 7 [ 2, 5, 6, 8 ], 8 [ 1, 2, 3, 4 ], \\
&&& 8 [ 1, 2, 5, 6 ]\\
252 & 1/2 & 69 & 1 [ 2, 7, 8 ], 8 [ 1, 2, 7, 8 ], 9 [ 2, 3, 6, 7, 8 ], 9 [ 2, 4, 5, 7, 8 ]\\
252 & 1 & 241 & 2 [ 2, 3, 5, 8 ], 2 [ 2, 4, 6, 7 ]\\
253 & $\infty$ & 3 & 1 [ 2, 5, 6 ], 1 [ 3, 6, 8 ], 1 [ 4, 6, 7 ], 6 [ 2, 3, 5, 6, 8 ], \\
&&& 6 [ 2, 4, 5, 6, 7 ], 6 [ 3, 4, 6, 7, 8 ], 7 [ 1, 4, 6, 7 ], \\
&&& 8 [ 1, 2, 5, 6 ], 8 [ 1, 3, 6, 8 ]\\
253 & -1/2 & 245 & 2 [ 1, 4, 5, 8 ]\\
253 & -2 & 240 & 2 [ 2, 4, 6, 8 ], 2 [ 5, 6, 7, 8 ]\\
253 & $\frac12(-3+\sqrt{5})$ & 453 & 2 [ 4, 5, 6, 8 ]\\
253 & $\frac12(-3-\sqrt{5})$ & 453 & 2 [ 4, 5, 6, 8 ]\\
254 & -1 & 241 & 2 [ 2, 4, 6, 7 ], 2 [ 4, 5, 7, 8 ]\\
254 & $\frac12(-1-\sqrt{5})$ & 453 & 2 [ 2, 4, 5, 7 ]\\
254 & $\frac12(-1+\sqrt{5})$ & 453 & 2 [ 2, 4, 5, 7 ]\\
255 & -1 & 32 & 1 [ 2, 5, 6 ], 1 [ 4, 6, 7 ], 6 [ 2, 4, 5, 6, 7 ], 7 [ 1, 4, 6, 7 ], \\
&&& 7 [ 2, 3, 5, 6 ], 7 [ 2, 5, 6, 8 ], 7 [ 3, 4, 6, 7 ], 8 [ 1, 2, 5, 6 ], \\
&&& 8 [ 4, 6, 7, 8 ]\\
255 & $\frac12(-1-\sqrt{5})$ & 453 & 2 [ 2, 3, 6, 7 ]\\
255 & $\frac12(-1+\sqrt{5})$ & 453 & 2 [ 2, 3, 6, 7 ]\\
256 & 1 & 239 & 2 [ 1, 4, 5, 8 ], 2 [ 2, 3, 6, 7 ]\\
256 & -1 & 238 & 2 [ 1, 4, 6, 7 ], 2 [ 2, 4, 6, 8 ], 2 [ 3, 4, 7, 8 ], 2 [ 5, 6, 7, 8 ]\\
256 & $\frac12(-1+\sqrt{-3})$ & 452 & 2 [ 4, 6, 7, 8 ]\\
256 & $\frac12(-1+\sqrt{-3})$ & 452 & 2 [ 4, 6, 7, 8 ]\\
257 & 1/2 & 240 & 2 [ 1, 4, 5, 8 ], 2 [ 2, 3, 6, 7 ]\\
257 & $\frac12(1+\sqrt{-3})$ & 452 & 2 [ 2, 3, 5, 8 ]\\
257 & $\frac12(1-\sqrt{-3})$ & 452 & 2 [ 2, 3, 5, 8 ]\\
258 & $\infty$ & 32 & 1 [ 2, 7, 8 ], 1 [ 3, 6, 8 ], 6 [ 2, 3, 6, 7, 8 ], 7 [ 2, 4, 7, 8 ], \\
&&& 7 [ 2, 5, 7, 8 ], 7 [ 3, 4, 6, 8 ], 7 [ 3, 5, 6, 8 ], 8 [ 1, 2, 7, 8 ], \\
&&& 8 [ 1, 3, 6, 8 ]\\
258 & 1 & 245 & 2 [ 2, 3, 5, 8 ]\\
258 & -1 & 93 & 1 [ 4, 5, 8 ], 7 [ 2, 4, 5, 8 ], 7 [ 3, 4, 5, 8 ], 8 [ 1, 4, 5, 8 ], \\
&&& 9 [ 4, 5, 6, 7, 8 ]\\
258 & -1/2 & 240 & 2 [ 2, 4, 6, 8 ], 2 [ 3, 4, 7, 8 ]\\
259 & $\infty$ & 32 & 1 [ 2, 5, 6 ], 1 [ 4, 6, 7 ], 6 [ 2, 4, 5, 6, 7 ], 7 [ 1, 4, 6, 7 ], \\
&&& 7 [ 2, 3, 5, 6 ], 7 [ 2, 5, 6, 8 ], 7 [ 4, 6, 7, 8 ], 8 [ 1, 2, 5, 6 ], \\
&&& 8 [ 3, 4, 6, 7 ]\\
259 & -1/2 & 32 & 1 [ 2, 7, 8 ], 1 [ 4, 5, 8 ], 6 [ 2, 4, 5, 7, 8 ], 7 [ 1, 4, 5, 8 ], \\
&&& 7 [ 2, 3, 7, 8 ], 7 [ 2, 6, 7, 8 ], 7 [ 4, 5, 6, 8 ], 8 [ 1, 2, 7, 8 ], \\
&&& 8 [ 3, 4, 5, 8 ]\\
262 & 1 & 1 & 1 [ 1, 5, 6 ], 1 [ 2, 3, 4 ], 1 [ 3, 6, 7 ], 1 [ 4, 5, 8 ], \\
&&& 6 [ 1, 3, 5, 6, 7 ], 6 [ 1, 4, 5, 6, 8 ], 6 [ 2, 3, 4, 5, 8 ], \\
&&& 6 [ 2, 3, 4, 6, 7 ], 7 [ 3, 6, 7, 8 ], 8 [ 1, 2, 3, 4 ], 8 [ 1, 2, 5, 6 ], \\
&&& 8 [ 4, 5, 7, 8 ]\\
262 & $\frac12(1-\sqrt{-3})$ & 451 & 1 [ 2, 5, 6 ], 7 [ 2, 4, 5, 6 ], 7 [ 2, 5, 6, 7 ], 8 [ 1, 2, 5, 6 ], \\
&&& 9 [ 2, 3, 5, 6, 8 ]\\
262 & $\frac12(1+\sqrt{-3})$ & 451 & 1 [ 2, 5, 6 ], 7 [ 2, 4, 5, 6 ], 7 [ 2, 5, 6, 7 ], 8 [ 1, 2, 5, 6 ], \\
&&& 9 [ 2, 3, 5, 6, 8 ]\\
265 & 1 & 69 & 1 [ 1, 5, 6 ], 8 [ 1, 2, 5, 6 ], 9 [ 1, 3, 5, 6, 7 ], 9 [ 1, 4, 5, 6, 8 ]\\
265 & -1 & 69 & 1 [ 2, 5, 6 ], 8 [ 1, 2, 5, 6 ], 9 [ 2, 3, 5, 6, 8 ], 9 [ 2, 4, 5, 6, 7 ]\\
266 & $\infty$ & 19 & 1 [ 2, 5, 6 ], 1 [ 3, 6, 8 ], 6 [ 2, 3, 5, 6, 8 ], 7 [ 3, 4, 6, 8 ], \\
&&& 7 [ 3, 6, 7, 8 ], 8 [ 1, 2, 5, 6 ], 8 [ 1, 3, 6, 8 ], 9 [ 2, 4, 5, 6, 7 ]\\
266 & 1/2 & 245 & 2 [ 3, 4, 7, 8 ]\\
266 & 1 & 19 & 1 [ 2, 7, 8 ], 1 [ 3, 5, 7 ], 6 [ 2, 3, 5, 7, 8 ], 7 [ 3, 4, 5, 7 ], \\
&&& 7 [ 3, 5, 6, 7 ], 8 [ 1, 2, 7, 8 ], 8 [ 1, 3, 5, 7 ], 9 [ 2, 4, 6, 7, 8 ]\\
266 & -1 & 245 & 2 [ 3, 4, 5, 6 ]\\
266 & 2 & 245 & 2 [ 5, 6, 7, 8 ]\\
268 & $\infty$ & 69 & 1 [ 2, 3, 4 ], 8 [ 1, 2, 3, 4 ], 9 [ 2, 3, 4, 5, 8 ], 9 [ 2, 3, 4, 6, 7 ]\\
268 & $\frac12(3+\sqrt{5})$ & 453 & 2 [ 2, 4, 5, 7 ]\\
268 & $\frac12(3-\sqrt{5})$ & 453 & 2 [ 2, 4, 5, 7 ]\\
273 & 0 & 19 & 1 [ 1, 3, 4 ], 1 [ 4, 5, 7 ], 6 [ 1, 3, 4, 5, 7 ], 7 [ 2, 4, 5, 7 ], \\
&&& 7 [ 4, 5, 7, 8 ], 8 [ 1, 2, 3, 4 ], 8 [ 4, 5, 6, 7 ], 9 [ 1, 3, 4, 6, 8 ]\\
273 & $\infty$ & 19 & 1 [ 2, 3, 4 ], 1 [ 4, 6, 7 ], 6 [ 2, 3, 4, 6, 7 ], 7 [ 1, 4, 6, 7 ], \\
&&& 7 [ 4, 6, 7, 8 ], 8 [ 1, 2, 3, 4 ], 8 [ 4, 5, 6, 7 ], 9 [ 2, 3, 4, 5, 8 ]\\
273 & 1/2 & 245 & 2 [ 1, 4, 5, 8 ]\\
273 & 1 & 19 & 1 [ 1, 2, 4 ], 1 [ 4, 5, 6 ], 6 [ 1, 2, 4, 5, 6 ], 7 [ 3, 4, 5, 6 ], \\
&&& 7 [ 4, 5, 6, 8 ], 8 [ 1, 2, 3, 4 ], 8 [ 4, 5, 6, 7 ], 9 [ 1, 2, 4, 7, 8 ]\\
273 & -1 & 245 & 2 [ 3, 4, 7, 8 ]\\
273 & 2 & 245 & 2 [ 2, 4, 6, 8 ]\\
274 & 0 & 3 & 1 [ 1, 3, 4 ], 1 [ 4, 5, 7 ], 1 [ 4, 6, 8 ], 6 [ 1, 3, 4, 5, 7 ], \\
&&& 6 [ 1, 3, 4, 6, 8 ], 6 [ 4, 5, 6, 7, 8 ], 7 [ 2, 4, 6, 8 ], \\
&&& 8 [ 1, 2, 3, 4 ], 8 [ 2, 4, 5, 7 ]\\
274 & 1 & 245 & 2 [ 1, 4, 5, 8 ]\\
274 & -1 & 3 & 1 [ 2, 5, 6 ], 1 [ 3, 6, 8 ], 1 [ 4, 6, 7 ], 6 [ 2, 3, 5, 6, 8 ], \\
&&& 6 [ 2, 4, 5, 6, 7 ], 6 [ 3, 4, 6, 7, 8 ], 7 [ 1, 4, 6, 7 ], \\
&&& 8 [ 1, 2, 5, 6 ], 8 [ 1, 3, 6, 8 ]\\
274 & -2 & 245 & 2 [ 2, 3, 6, 7 ]\\
274 & $\frac12(-1-\sqrt{-3})$ & 452 & 2 [ 3, 4, 5, 6 ]\\
274 & $\frac12(-1+\sqrt{-3})$ & 452 & 2 [ 3, 4, 5, 6 ]\\
454 & 1 & 451 & 1 [ 2, 5, 6 ], 7 [ 2, 3, 5, 6 ], 7 [ 2, 5, 6, 8 ], 8 [ 1, 2, 5, 6 ], \\
&&& 9 [ 2, 4, 5, 6, 7 ]\\
454 & $\frac12(1-\sqrt{-3})$ & 451 & 1 [ 2, 7, 8 ], 7 [ 2, 4, 7, 8 ], 7 [ 2, 5, 7, 8 ], 8 [ 1, 2, 7, 8 ], \\
&&& 9 [ 2, 3, 6, 7, 8 ]\\
\hline
\end{longtable}

\printbibliography
\end{document}